\documentclass[reqno]{amsart}
\usepackage[dvipsnames]{xcolor}
\usepackage{comment}
\usepackage{amsmath,amsthm,amssymb,relsize,mathrsfs}
\usepackage[nobysame,alphabetic]{amsrefs}
\usepackage[margin=35mm, marginpar=25mm]{geometry}

\usepackage{booktabs, multirow, adjustbox, array}

\usepackage{pifont}

\usepackage{enumitem}
\usepackage{tikz}
\usetikzlibrary{
  cd,
  calc,
  positioning,
  fit,
  arrows,
  decorations.pathreplacing,
  decorations.markings,
  shapes.geometric,
  backgrounds,
  bending
}
\usepackage{tikzsymbols}
\usepackage[arrow,matrix,curve,cmtip,ps]{xy}
\usepackage{pb-diagram}
\usepackage{pb-xy}

\definecolor{darkblue}{rgb}{0,0,0.6}

\usepackage[breaklinks, pdftex, ocgcolorlinks,colorlinks=true, citecolor=darkblue, filecolor=darkblue, linkcolor=darkblue, urlcolor=darkblue]{hyperref}
\usepackage[capitalize,noabbrev]{cleveref}

\newtheorem{proposition}{Proposition}[section]
\newtheorem{theorem}[proposition]{Theorem}
\newtheorem{corollary}[proposition]{Corollary}
\newtheorem{lemma}[proposition]{Lemma}

\newtheorem{thmx}{Theorem}

\theoremstyle{definition}
\newtheorem{definition}[proposition]{Definition}
\newtheorem{question}[proposition]{Question}

\newtheorem{notation}[proposition]{Notation}

\theoremstyle{remark}
\newtheorem{remark}[proposition]{Remark}

\newtheorem*{remark*}{Remark}

\AddToHook{env/theorem/begin}{\crefalias{proposition}{theorem}}
\AddToHook{env/thmx/begin}{\crefalias{proposition}{thmx}}
\AddToHook{env/corx/begin}{\crefalias{proposition}{corx}}
\AddToHook{env/assumption/begin}{\crefalias{proposition}{assumption}}
\AddToHook{env/hypothesis/begin}{\crefalias{proposition}{hypothesis}}
\AddToHook{env/corollary/begin}{\crefalias{proposition}{corollary}}
\AddToHook{env/definition/begin}{\crefalias{proposition}{definition}}
\AddToHook{env/lemma/begin}{\crefalias{proposition}{lemma}}
\AddToHook{env/question/begin}{\crefalias{proposition}{question}}
\AddToHook{env/example/begin}{\crefalias{proposition}{example}}
\AddToHook{env/conjecture/begin}{\crefalias{proposition}{conjecture}}
\AddToHook{env/remark/begin}{\crefalias{proposition}{remark}}
\AddToHook{env/const/begin}{\crefalias{proposition}{const}}
\AddToHook{env/construction/begin}{\crefalias{proposition}{construction}}

\crefname{assumption}{Assumption}{Assumptions}
\crefname{hypothesis}{Hypothesis}{Hypotheses}
\crefname{theorem}{Theorem}{Theorems}
\crefname{thmx}{Theorem}{Theorems}
\crefname{proposition}{Proposition}{Propositions}
\crefname{corollary}{Corollary}{Corollaries}
\crefname{definition}{Definition}{Definitions}
\crefname{lemma}{Lemma}{Lemmas}
\crefname{question}{Question}{Questions}
\crefname{example}{Example}{Examples}
\crefname{conjecture}{Conjecture}{Conjectures}
\crefname{remark}{Remark}{Remarks}
\crefname{const}{Construction}{Constructions}
\crefname{construction}{Construction}{Constructions}

\newcommand{\Q}{\mathbb{Q}}
\newcommand{\N}{\mathbb{N}}

\newcommand{\Z}{\mathbb{Z}}

\newcommand{\PD}{\mathrm{PD}}
\newcommand{\im}{\operatorname{Im}}
\newcommand{\Id}{\operatorname{Id}}

\newcommand{\ol}{\overline}

\newcommand{\wt}{\widetilde}
\newcommand{\wh}{\widehat}
\newcommand{\sm}{\setminus}

\newcommand{\ks}{\mathrm{ks}}
\newcommand{\KS}{\mathsf{KS}}
\newcommand{\cs}{\mathrm{cs}}

\DeclareMathOperator{\Wh}{Wh}

\DeclareMathOperator{\pt}{pt}

\DeclareMathOperator{\coker}{coker}

\newcommand{\CAT}{\mathrm{CAT}}
\newcommand{\Top}{\mathrm{Top}}
\newcommand{\Diff}{\mathrm{Diff}}

\newcommand{\SO}{\mathrm{SO}}

\newcommand{\pseudoHomeo}{\wt{\operatorname{Homeo}}\mkern 0mu}

\newcommand{\xratop}{\xrightarrow{\cong_{C^0}}}
\newcommand{\xrasmooth}{\xrightarrow{\cong_{C^{\infty}}}}

\begin{document}
\title[Smooth stable isotopy of topologically isotopic surfaces]{Smooth stable isotopy of topologically isotopic surfaces}

\author[D.~Galvin]{Daniel Galvin}
	\address{Department of Mathematics, University of Texas at Austin, USA}
	\email{daniel.galvin@austin.utexas.edu}

\author[P.~Orson]{Patrick Orson}
\address{Mathematics Department, California Polytechnic State University, USA}
\email{porson@calpoly.edu}

 \author[M.~Powell]{Mark Powell}
 \address{School of Mathematics and Statistics, University of Glasgow, United Kingdom}
 \email{mark.powell@glasgow.ac.uk}

\def\subjclassname{\textup{2020} Mathematics Subject Classification}
\expandafter\let\csname subjclassname@1991\endcsname=\subjclassname


\begin{abstract}
A stabilisation of a $4$-manifold  $X$ is the connected sum of $X$ with some number of copies of $S^2\times S^2$. If two smooth surfaces in a $4$-manifold are topologically isotopic, we investigate whether they must moreover be smoothly isotopic in some stabilisation of $X$.
We prove this result holds whenever the surfaces are trivial in the $\mathbb{Z}/2$-homology of $X$. We also produce a large class of fundamental groups of the ambient $4$-manifold for which  the result holds; this class includes free products of classical knot groups and, in particular, free groups.
\end{abstract}
\maketitle

\section{Introduction}
\label{sec:intro}

A \emph{stabilisation} of a smooth, compact, connected, orientable $4$-manifold $X$ is the effect of taking the connected sum of $X$ with some number of copies of~$S^2\times S^2$. Topological results about $4$-manifolds sometimes become true smoothly, after stabilisation. In the most famous instance, results of Wall \cites{MR163323,MR163324} 
in the simply-connected case
and Gompf~\cite{Gompf-stable}, generally, show that homeomorphic closed, orientable, smooth $4$-manifolds are stably diffeomorphic.

In some cases, a similar phenomenon is known to occur for isotopy classes of self-maps of a~$4$-manifold. Given $g\geq 1$ and a diffeomorphism $f\colon X\to X$, one may assume after isotopy that~$f$ fixes a $4$-ball, use that to form the connected sum $X\#g(S^2\times S^2)$, then extend~$f$ by the identity, to get a diffeomorphism of this stabilisation of $X$. We call this a \emph{stabilisation} of~$f$. When~$X$ is closed and simply-connected, Kreck~\cite{Kreck-isotopy-classes} 
and Quinn~\cite{Quinn:isotopy} (cf.~\cite{GGHKP}) showed that topologically isotopic diffeomorphisms of $X$ are stably smoothly isotopic. Results of Krannich--Kupers~\cite{KK24} and Gabai~\cite{Gabai-22} combine to extend this to the case of free fundamental group, and Orson--Powell--Randal-Williams~\cite{OPRW} further extend this to a larger class of fundamental groups. It remains open whether this result holds for every compact $4$-manifold.

Here is a natural variant of the stabilisation question for embedded surfaces.

\begin{question}\label{q:main}
    If two smooth surfaces $\Sigma_1,\Sigma_2\subseteq X$ are topologically isotopic, are they necessarily smoothly isotopic in some stabilisation of $X$?
\end{question}

 The first named author provided a positive answer in the case that $\pi_1(X)$ is trivial~\cite{GalvinCS}*{Theorem~1.2}.
Our main result gives a positive answer for some cases in which $X$ is not simply-connected.  The question remains open, in general.

 Let $X$ be a smooth, orientable, compact $4$-manifold and write $\pi: = \pi_1(X)$. 
    Suppose $\Sigma_1, \Sigma_2\subseteq X$ are smooth, compact, proper surfaces with $\partial \Sigma_1 = \partial \Sigma_2$.

\begin{thmx}\label{thm:main}
       Suppose  
       at least one of the following holds.
    \begin{enumerate}[label=(\roman*)]

        \item\label{thm:main-item-iii} For every connected component $\Sigma_1^j$ of $\Sigma_1$, and for every $x \in H_2(X;\Z/2)$, we have that~$\lambda^{\Z/2}(x,[\Sigma_1^j]) = 0 \in \Z/2$.

        \item\label{thm:main-item-iv} 
        The component $I_2 \colon H_2(\pi;\Z_{(2)}) \to L_6(\Z\pi)_{(2)}$ of the algebraic assembly map is zero, there is no 2-torsion in $H_1(\pi;\Z)$, and $\Wh_2(\pi)=0$.
    \end{enumerate}
  Then if $\Sigma_1$ and $\Sigma_2$ are topologically isotopic rel.\ boundary,  they
 are smoothly isotopic rel.\ boundary in some stabilisation of $X$.
\end{thmx}

    In~\cref{thm:main}, we are not assuming $X$ is closed, and we are not assuming the surfaces are orientable. Observe that in \ref{thm:main-item-iii}, even if $\Sigma_1^j$ is closed, we can still consider it as a homology class in $H_2(X,\partial X;\Z/2)$. We write 
    \[\lambda^{\Z/2} \colon H_2(X;\Z/2) \times H_2(X,\partial X;\Z/2) \to \Z/2\] 
    for the $\Z/2$-intersection pairing of $X$. Note that the condition in \ref{thm:main-item-iii} holds automatically if each connected component of $\Sigma_1$ is trivial in $H_2(X,\partial X;\Z/2)$. 

   The  technical hypotheses on the fundamental group  in \cref{thm:main}~\ref{thm:main-item-iv} arise from an application of~\cite{OPRW}*{Corollary~E}. As discussed in~\cite{OPRW}*{loc.~ cit.}, these conditions on $\pi$ are known to hold for free products of knot groups.

We record these observations in the following corollary to \cref{thm:main}. 

\begin{corollary}\label{cor:main}
Suppose at least one of the following holds.    
        \begin{enumerate}[label=(\roman*)]
        \item\label{cor:item-iii} For every connected component $\Sigma_1^j$ of $\Sigma_1$, we have that 
        $[\Sigma_1^j] = 0 \in H_2(X,\partial X;\Z/2)$.
        \item\label{cor:item-iv} 
        The fundamental group $\pi$ is a free product of classical knot groups. 
           \end{enumerate}
            Then  if $\Sigma_1$ and $\Sigma_2$ are topologically isotopic rel.\ boundary,  they are smoothly isotopic rel.\ boundary in some stabilisation of $X$.

\end{corollary}

\begin{remark}

Cha--Kim~\cite{chakim-lightbulb} showed that given a locally flat surface $\Sigma\subseteq X$ there exists some stabilisation of $X$ in which $\Sigma$ is topologically isotopic to a smooth surface. The Cha--Kim result can be viewed as a \emph{stable existence} statement and~\cref{thm:main} and \cref{cor:main} can be viewed as  \emph{stable uniqueness} counterparts, where they apply. 
\end{remark}

\subsection*{Outline of the proof}

Here is a brief summary of the strategy for the proof, which doubles as a description of the organisation of the paper. 
\begin{itemize}[leftmargin=*]
    \item In \cref{section:CS-invariant}, we begin with a  homeomorphism of $X$ sending $\Sigma_1$ to $\Sigma_2$, that is topologically isotopic rel.~boundary to the identity. We show, following~\cite{GalvinCS}, that such a homeomorphism is always smoothable near~$\Sigma_1$. Again following~\cite{GalvinCS}, we introduce the Casson-Sullivan obstruction to stably smoothing the homeomorphism on the exterior of $\Sigma_1$.  We prove a naturality property and  analyse the obstruction in our surface exterior case.  
    \item In \cref{section:killing-CS}, we modify the given homeomorphism to a \emph{stably} smoothable one, and thus achieve a diffeomorphism sending~$\Sigma_1$ to~$\Sigma_2$ in some stabilisation of $X$. This is achieved by use of the first-named author's stable realisation of the Casson-Sullivan invariant~\cite{GalvinCS}. 
    \item In \cref{section:obtaining-diffeo-top-PI-to-Id}, we show that this diffeomorphism can be chosen in such a way that it is topologically pseudo-isotopic to the identity (\cref{defn:pseudo-isotopy}).
    \item In \cref{section:improving-top-PI-to-smooth-PI}, we consider conditions under which the topological pseudo-isotopy can be replaced by a smooth pseudo-isotopy. For this we analyse a smoothing obstruction for topological pseudo-isotopies introduced by Orson--Powell--Randal-Williams~\cite{OPRW}.  The hypotheses \ref{thm:main-item-iii} and \ref{thm:main-item-iv} of \cref{thm:main} are used in this section.  
    \item In \cref{section:completing-the-proof}, we complete the proof. Following a careful analysis involving the $\Wh_2$ group, we are able to appeal to a result of Singh~\cite{Singh} to obtain a diffeomorphism of $X$ sending $\Sigma_1$ to~$\Sigma_2$, that by a theorem of Gabai~\cite{Gabai-22}*{Theorem~2.5} is smoothly stably isotopic rel.~boundary to the identity. 
\end{itemize}

\subsection*{Discussion of the proof}

    The first named author's proof \cite{GalvinCS}*{Theorem~1.2} of our main result in the case that $X$ is simply connected effectively relies on the fact that diffeomorphisms of simply-connected 4-manifolds are classified up to smooth pseudo-isotopy~\cites{Kreck-isotopy-classes, Saeki, Orson-Powell}. 
    Outside of the simply connected setting, such a classification is only known when $\pi_1(X) \cong \Z$, topologically, by Stong--Wang~\cite{Stong-Wang}, and smoothly, by combining \cite{Stong-Wang} with \cite{OPRW}. Thus this proof strategy is quite limited, given the current state of knowledge.
    
    Despite this, our main result allows many more fundamental groups. Indeed, \cref{thm:main}~\ref{thm:main-item-iii} holds for completely general finitely presented fundamental groups. So the most surprising aspect of our proof is that we are able to circumvent this lack of a pseudo-isotopy classification and achieve the outcome of \cref{section:obtaining-diffeo-top-PI-to-Id}: the existence of a diffeomorphism sending~$\Sigma_1$ to~$\Sigma_2$, in some stabilisation of $X$, that is topologically pseudo-isotopic to the identity.
    
    How is this achieved?  One starts with a homeomorphism of $X$ sending $\Sigma_1$ to $\Sigma_2$ that is topologically isotopic to the identity and wishes to modify the homeomorphism to smooth it to a diffeomorphism, at the expense of stabilising. The obstruction is the Casson--Sullivan invariant, and the danger is that in modifying the homeomorphism to kill this obstruction, we lose control of the topological pseudo-isotopy class of the eventual diffeomorphism, which should remain topologically pseudo-isotopic to the identity (or else we have essentially jettisoned the initial hypothesis of the proof). The key technical step to retain control is \cref{prop:gammas}, in which we show that, to kill the Casson--Sullivan invariant, it is only necessary to modify the original homeomorphism in a neighbourhood of a union of meridians of $\Sigma_1$. As a consequence, the change can be localised to a simply connected codimension zero submanifold of~$X \# W_k$,  whereupon we can apply the Orson-Powell~\cite{Orson-Powell} topological isotopy classification for homeomorphisms of simply-connected 4-manifolds with boundary. We thus obtain a topological isotopy from the original homeomorphism, and hence also from the identity, to the one with trivial Casson--Sullivan invariant.  
    The latter is stably topologically pseudo-isotopic to a diffeomorphism, leading to the desired stable topological pseudo-isotopy from this diffeomorphism to the identity.
    
 \subsection*{Remark on a special case}
      As mentioned above, the case of Theorem~\ref{thm:main} when $\pi_1(X)=\{1\}$ was first proved in~\cite{GalvinCS}*{Theorem~1.2}. In the further special case of $\pi_1(X \sm \Sigma_i) = \{1\}$ (which implies~$\pi_1(X)=1$), there is an alternative proof that uses work of Gompf~\cite{Gompf-stable}, Boyer~\cite{Boyer93}, Saeki~\cite{Saeki}, and Quinn~\cite{Quinn:isotopy} (with the correction in~\cite{GGHKP}). We detail this proof in Appendix~\ref{sec:appendix}, for the interested reader.

      The existence of such an alternative proof, based on existing results, was suggested in the introduction of~\cite{AKMR}, where it was asserted that the combination of work of Wall~\cite{MR163323}, Perron~\cite{Perron}, and Quinn~\cite{Quinn:isotopy} would show homologous $2$-spheres with simply-connected complement are topologically isotopic, and become smoothly isotopic in some stabilisation. Such a proof was again suggested, this time for general surfaces,  in the introduction of~\cite{AKMRS}. Here it was asserted that the result would follow from Perron~\cite{Perron} and Quinn~\cite{Quinn:isotopy}.
      In neither case were any details given. In Remark~\ref{rem:AKMR}, we discuss that we do not believe such a proof would work using these citations, particularly without applying the results of Boyer~\cite{Boyer93} and Saeki~\cite{Saeki}.  

\subsubsection*{Conventions}

The following conventions and notation are used throughout.

\begin{itemize}[leftmargin=*]
    \item For the remainder of the article, fix  a smooth, orientable, compact $4$-manifold $X$, and denote~$\pi: = \pi_1(X)$. 
    \item A compact submanifold $\Sigma \subseteq X$ is \emph{proper} if the inclusion map $\iota \colon \Sigma \to X$ satisfies $\iota^{-1}(\partial X) = \partial \Sigma$. Note this adjective is meaningful when either of $\partial \Sigma$ and $\partial X$ are empty. 
    \item The symbol $\nu$ applied to a subset $C\subseteq X$ of a topological space means any open neighbourhood. If $C\subseteq X$ is moreover a submanifold, we will use $\nu C$ to specifically mean an open tubular neighbourhood and $\overline{\nu} C$ to denote a closed tubular neighbourhood. If $C=\partial X$, then~$\nu C$ and $\overline{\nu}C$ denote open and closed boundary collars, respectively.
    \item For the remainder of the article, fix $\Sigma_1, \Sigma_2\subseteq X$ smooth, proper surfaces with $\partial \Sigma_1 = \partial \Sigma_2$. The surfaces are permitted to be nonorientable.  For~$i=1,2$, denote the surface exterior by~$X_i:=X\sm \nu\Sigma_i$.
    \item Given a space $A$ and subspace $B\subseteq A$, a homotopy $\Psi\colon A\times I\to A$ is \emph{rel.~$B$} if $\Psi(x,t)=\Psi(x,0)$ for all $x\in B$ and $t\in I$.
    \item We denote the connected sum of $n$ copies of $S^2\times S^2$ by $W_n=n(S^2\times S^2)$.
\end{itemize}

\subsubsection*{Acknowledgements}

We thank the Centre de Recherche Math\'ematiques at the Universit\'{e} de Montr\'{e}al for hospitality during the 2025 thematic programme on `Topological and Geometric structures in low dimensions', during which part of this paper was written. 
MP was a CRM-Simons Visiting Professor during this programme, and is grateful for the associated support. 
DG thanks the Max Planck Institute for Mathematics for support. 
Part of this research was also done while MP and DG were visiting the MPIM in Bonn. We warmly thank Simona Vesel\'{a} for very helpful discussions about the result in the appendix.

\section{The Casson-Sullivan invariant for a homeomorphism of surface exteriors}\label{section:CS-invariant}

In this section we recall the Casson-Sullivan invariant; this is the relative Kirby-Siebenmann invariant of a homeomorphism~\cite{Kirby-Siebenmann:1977-1} and is an obstruction to smoothing the homeomorphism. The terminology ``Casson-Sullivan invariant'' arose from the importance of this, as a triangulation obstruction, to the Manifold Hauptvermutung~\cites{MR1434102, MR1434103}. In the context of $4$-manifolds, the invariant was investigated by the first-named author in~\cite{GalvinCS}. We discuss some elementary properties and then analyse the Casson-Sullivan invariant for a homeomorphism between surface exteriors that restricts to a diffeomorphism on the boundary.

\subsection{The Casson-Sullivan invariant}
\label{sec:csbasics}

We begin by recalling the general definition of the Casson-Sullivan invariant for $4$-manifolds, and recap some elementary naturality properties.

\medskip

For $i=1,2$, let $M_i$ be a smooth $4$-manifold and let $C_i\subseteq M_i$ be a closed subset. Suppose
\[
F\colon (M_1,C_1)\xratop (M_2,C_1)
\]
is a homeomorphism of pairs that restricts to a diffeomorphism
\[
F|_{\nu C_1}\colon \nu C_1\xrasmooth\nu C_2
\]
for some open neighbourhood $\nu C_1$ of $C_1\subseteq M_1$. Let $\sigma_i$ denote the smooth structure on $M_i$. Consider $M_1 \times I$ and endow $(M_1\times [0,\varepsilon)) \cup (\nu C_1 \times I)$ with the smooth structure $(\sigma_1\times \mathrm{std}_{[0,\varepsilon)}) \cup (\sigma_1 \times \mathrm{std}_{I})$.  Endow $M_1 \times (1-\varepsilon,1]$ with the pullback smooth structure $F^*\sigma_2\times\mathrm{std}_{(1-\varepsilon,1]}$. As $F$ restricts to a diffeomorphism on $\nu C_1$ these structures are compatible on the overlap; we denote by $\Sigma$ the resulting smooth manifold $(M_1\times[0,\varepsilon)\cup(1-\varepsilon,1])\cup (\nu C_1\times I)$. 
This smooth structure extends to all of $M_1 \times I$ if and only if the  
Kirby-Siebenmann obstruction
\[
\ks(M_1 \times I,\Sigma) \in H^4(M_1 \times I,\Sigma;\Z/2)
\]
vanishes. 
Excision and the long exact sequence of the triple $(M_1 \times I, \Sigma, \Sigma\setminus (M_1\times[0,\varepsilon))$ yield isomorphisms 
\begin{equation}
    \label{eq:excision}
H^3(M_1,C_1;\Z/2) \cong H^3(\Sigma, \Sigma\setminus (M_1\times[0,\varepsilon));\Z/2) \xrightarrow{\cong} H^4(M_1 \times I, \Sigma;\Z/2).
\end{equation}
We explain the two isomorphisms. 
\begin{itemize}[leftmargin=*]
    \item  For the first map, we consider the maps of pairs 
    \[(M_1,C_1) \leftarrow ( M_1 \times [0,\varepsilon) \cup (C_1 \times I), C_1 \times [\varepsilon,1]) \to (\Sigma, \Sigma\setminus (M_1\times[0,\varepsilon)). \]
The leftwards map is by definition projection, and is a homotopy equivalence of pairs. The rightwards map is an inclusion of pairs, and corresponds to excising $(M_1 \sm C_1) \times (1-\varepsilon,1]$. Thus both maps induce isomorphisms on cohomology, leading to the left isomorphism in~\eqref{eq:excision}. 
\item The second map is an isomorphism because, in the long exact sequence of the triple, the other terms are 
\[H^k(M_1 \times I, \Sigma\setminus (M_1\times[0,\varepsilon);\Z/2) \cong H^k(M_1 \times I, M_1;\Z/2) =0\]
for $k=3,4$. 
\end{itemize}

\begin{definition}
 With notation as above, the \emph{Casson-Sullivan invariant of $F$ relative to $\nu C_1$} is the element
 \[
 \cs(\text{$F$ rel.~$\nu C_1$})\in H^3(M_1, C_1;\Z/2)
 \]
 mapping to $\ks(M_1 \times I,\Sigma)$ under the sequence of isomorphisms~\eqref{eq:excision}.

 \end{definition}
  
  In the case that the subsets $C_i\subseteq M_i$ are submanifolds, the neighbourhoods $\nu C_i$ are by convention tubular neighbourhoods, and are thus unique. Similarly, if $C_i=\partial M_i$ then the neighbourhoods are unique, by uniqueness of boundary collars. So in these cases the resulting Casson-Sullivan invariant is independent of the choice of open neighbourhoods, and thus we can write
  \[
 \cs(\text{$F$ rel.~$C_1$}):= \cs(\text{$F$ rel.~$\nu C_1$})\in H^3(M_1, C_1;\Z/2).
 \]

  \begin{definition}
  When $F\colon M_1\to M_2$ is a homeomorphism restricting to a diffeomorphism on the boundary, we write
\[
\cs(F):=\cs(\text{$F$ rel.~$\partial M_1$})\in H^3(M_1, \partial M_1;\Z/2),
\]
and call this simply the \emph{Casson-Sullivan invariant of $F$}.
\end{definition}

We recall our notation for connected sums of $S^2\times S^2$ from the conventions in~\cref{sec:intro}.

\begin{definition}
        We denote the connected sum of $n$ copies of $S^2\times S^2$ by $W_n:=n(S^2\times S^2)$.
\end{definition}

In order to state a key property of the Casson-Sullivan invariant, we need the following definition. 

\begin{definition}\label{defn:pseudo-isotopy}
Let $\CAT \in \{\Diff,\Top\}$, let $M_1, M_2$ be compact, smooth 4-manifolds, and let~$F,G \colon M_1\to M_2$ be $\CAT$-isomorphisms that agree on the boundary $G|_{\partial M_1}=F|_{\partial M_1}$. 
\begin{enumerate}[label=(\roman*)]
    \item 
We say that $F$ and $G$ are \emph{$\CAT$ pseudo-isotopic} if there is a $\CAT$-isomorphism \[\Psi \colon M_1 \times I \xrightarrow{\cong_{\CAT}} M_2 \times I\] with $F=\Psi|_{M_1\times\{0\}}$, $G=\Psi|_{M_1\times\{1\}}$ and such that the restriction of $\Psi$ to $\partial M_1 \times I$ is the product isotopy $\Psi(x,t)=(F(x),t)$.
\item We say that $F$ and $G$ are \emph{$\CAT$ stably pseudo-isotopic} if there exists $k \geq 0$ and there exist stabilisations~$F\#\Id_{W_k}$ and~$G\#\Id_{W_k}$ that are $\CAT$ pseudo-isotopic $\CAT$-isomorphisms.

\end{enumerate}
\end{definition}

\begin{remark}
    Our conventions listed in~\cref{sec:intro} state that isotopies need not be rel.~boundary, and will always be explicitly specified to be so when needed. Note, in contrast, that pseudo-isotopies are \emph{always} rel.~boundary, in our definition. 
\end{remark}

Here is a key property of the Casson-Sullivan invariant that we will require later; for a proof see \cite{GalvinCS}*{Proposition 2.19, Proposition 2.23}.

\begin{proposition}[\cite{FQ}*{Theorem~8.6(2)}]\label{prop:key-property-of-CS}
    Let~$F \colon M_1\to M_2$ be homeomorphism of compact, smooth 4-manifolds, that restricts to a diffeomorphism on the boundary. We have that $\cs(F) =0$ if and only if $F$ is smoothly stably pseudo-isotopic to a diffeomorphism.
\end{proposition}

We will also use the following, essentially formal, naturality statement for the Casson--Sullivan invariant. 

\begin{lemma}
\label{lem:natCS}
  Let $F \colon (M_1, C_1) \to (M_2, C_2)$ be a homeomorphism of pairs, where for $i=1,2$,  $M_i$ is a smooth 4-manifold and $C_i$ is a closed subset.  Suppose that $F$ restricts to a diffeomorphism on some open neighbourhoods $\nu C_1\to \nu C_2$.
  Let $(V_i, D_i) \subseteq (M_i, C_i)$ be an inclusion of pairs where~$V_i$ is an open codimension zero submanifold and $D_i \subseteq C_i$ is also closed. Assume that~$F$ restricts to a homeomorphism $(V_1,D_1) \to (V_2,D_2)$. Write $\nu D_1\subseteq \nu C_1$ for some open neighbourhood and~$\nu D_2$ for its image under $F$. Then under the inclusion $(V_1,D_1)\subseteq (M_1, C_1)$ we have
  \[
  H^3(M_1,C_1)\to H^3(V_1, D_1);\qquad \cs(\text{$F$ rel.~$\nu C_1$})\mapsto \cs(\text{$F|_{V_1}$ rel.~$\nu D_1$}).
  \]
\end{lemma}

\begin{proof}
     Given $F \colon (M_1,C_1) \to (M_2,C_2)$, to define the Casson-Sullivan invariant $\cs(\text{$F$ rel.~$\nu C_1$})$, at the beginning of Section~\ref{sec:csbasics}, we first built a smooth structure on the open submanifold
     \[
     (M_1\times[0,\varepsilon)\cup(1-\varepsilon])\cup (\nu C_1\times I)\subseteq M_1\times I.
     \]
     In this proof we will denote that smooth manifold by $\Sigma_C\subseteq M_1\times I$. Following the similar construction for $F|_{V_1} \colon (V_1,D_1) \to (V_2,D_2)$, we obtain a smooth structure on the open submanifold 
     \[
     (V_1\times[0,\varepsilon)\cup(1-\varepsilon])\cup (\nu D_1\times I)\subseteq V_1\times I,
     \]
     which we denote by $\Sigma_D\subseteq V_1\times I$. We thus have inclusions of open submanifolds
     \[
     \begin{tikzcd}
         \Sigma_D\ar[r]\ar[d]
         & V_1\times I\ar[d]
         \\
         \Sigma_C\ar[r]
         & M_1\times I
     \end{tikzcd}
     \]
     As the relative Kirby-Siebenmann invariant is natural under such maps \cite{Kirby-Siebenmann:1977-1}*{Essay IV, Theorem 10.1}, this induces
       \[
  H^4(M_1\times I,\Sigma_C)\to H^4(V_1\times I,\Sigma_D);\qquad \ks(M_1\times I,\Sigma_C)\mapsto \ks(V_1\times I,\Sigma_D).
  \]
  Finally, the isomorphisms described in~\eqref{eq:excision} induce the corresponding isomorphisms for the pair~$(V_1\times I,\Sigma_D)$, upon restriction, so that the following commutes
  \[
  \begin{tikzcd}
      H^3(V_1,D_1;\Z/2)
      &
      \ar[l, "\cong"'] H^3(\Sigma_D, \Sigma_D\setminus (V_1\times[0,\varepsilon));\Z/2) \ar[r, "\cong"]
      & H^4(V_1 \times I, \Sigma_D;\Z/2)
      \\
      H^3(M_1,C_1;\Z/2)\ar[u]
      &
      \ar[l, "\cong"'] H^3(\Sigma_C, \Sigma_C\setminus (M_1\times[0,\varepsilon));\Z/2) \ar[r, "\cong"]\ar[u]
      & H^4(M_1 \times I, \Sigma_C;\Z/2)\ar[u]
  \end{tikzcd}
  \]
  Thus under the inclusion $(V_1,D_1)\subseteq (M_1, C_1)$ we have
  \[
  H^3(M_1,C_1)\to H^3(V_1, D_1);\qquad \cs(\text{$F$ rel.~$\nu C_1$})\mapsto \cs(\text{$F|_{V_1}$ rel.~$\nu D_1$}),
  \]
  which completes the proof.
\end{proof}

\subsection{Casson-Sullivan for the surface exteriors}

Recall that $X$ is a fixed smooth, orientable, compact $4$-manifold and $\Sigma_1, \Sigma_2\subseteq X$ are smooth, proper surfaces with $\partial \Sigma_1 = \partial \Sigma_2$.

\begin{definition}
\label{notation1}
We say a homeomorphism of pairs
\[
\widehat{F}\colon (X,\ol{\nu}\Sigma_1)\xrightarrow{\cong_{C^0, C^\infty}}
 (X,\ol{\nu}\Sigma_2)
 \]
 is \emph{smooth near $\Sigma_1$} if
  the restriction \[\wh{F}|_{\ol{\nu}\Sigma_1}\colon \ol{\nu}\Sigma_1\xrightarrow{\cong_{C^{\infty}}} \ol{\nu}\Sigma_2\] is a diffeomorphism sending $\Sigma_1$ to $\Sigma_2$. 
\end{definition}

\begin{notation}
Given a homeomorphism that is smooth near~$\Sigma_1$
    \[
\widehat{F}\colon (X,\ol{\nu}\Sigma_1)\xrightarrow{\cong_{C^0, C^\infty}}
 (X,\ol{\nu}\Sigma_2),
 \]
for $i=1,2$, write
\[
F\colon X_1\to X_2,\quad \text{where~$F:=\wh{F}|_{X_1}$}
\]
(recall that $X_i:=X\sm \nu\Sigma_i$).
Note that $F|_{\partial X_1}$ is a diffeomorphism.
\end{notation}

We now recall that a homeomorphism sending $\Sigma_1$ to $\Sigma_2$ can always be smoothed near the surfaces. For this we will use that the surfaces $\Sigma_i\subseteq X$ and their closed tubular neighbourhoods~$\ol{\nu}\Sigma_i$ are already smooth submanifolds of $X$.

\begin{lemma}\label{lem:make-homeo-smooth-on-surface-neighbourhood}
Suppose we are given a homeomorphism
\[
\wh{F}'\colon X\xrightarrow{\cong_{C^0}} X
\]
such that $\wh{F}'(\Sigma_1)=(\Sigma_2)$.
    Then $\wh{F}'$ is topologically isotopic rel.~boundary to a homeomorphism that is smooth near~$\Sigma_1$
    \[
\widehat{F}\colon (X,\ol{\nu}\Sigma_1)\xrightarrow{\cong_{C^0, C^\infty}}
 (X,\ol{\nu}\Sigma_2),
 \]
\end{lemma}

\begin{proof}
This was proved in \cite{GalvinCS}*{Lemma~5.3}, and we refer the reader to there for full details. We provide a sketch for convenience. 
Smooth~$\wh{F}'\vert_{\Sigma_1}$ using that homeomorphisms of surfaces are smoothable, and extend using isotopy extension. By uniqueness of tubular neighbourhoods~\cite{FQ}*{Theorem 9.3}, we can further isotope to a homeomorphism sending a smooth tubular neighbourhood of~$\Sigma_1$ to a smooth tubular neighbourhood of~$\Sigma_2$, via a $D^2$-bundle (with $\SO(2)$ structure group) isomorphism covering a diffeomorphism.  As shown in \cite{GalvinCS}*{Lemma~5.3}, such a bundle isomorphism can be fibrewise isotoped so as to give a diffeomorphism between the total spaces.
\end{proof}

We will need the following notation for the cornered structure of the surface exteriors.

\begin{notation}\label{not:decomposition}
For $i=1,2$, the exterior $X_i$ is a manifold with corners, where
\[
\partial X_i=\partial_0X_i\cup_{\partial_{01}X_i}\partial_1X_i,\qquad\text{defined by}\qquad \partial_0X_i:=\partial X_i\cap \partial X,
\]
so that $\partial_1X_i$ is the total space of the sphere bundle of the normal bundle to $\Sigma_i$.
The total space of the disc bundle to $\Sigma_i$ is also a manifold with corners, where
\[
\partial \ol{\nu}\Sigma_i=\partial_0\ol{\nu}\Sigma_i\cup_{\partial_{01}\ol{\nu}\Sigma_i}\partial_1\ol{\nu}\Sigma_i,\qquad \text{defined by}\qquad \partial_0\ol{\nu}\Sigma_i:=\partial \ol{\nu}\Sigma_i\cap \partial X,
\]
so that $\partial_1\ol{\nu}\Sigma_i=\partial_1X_i$ and $\partial_{01}\ol{\nu}\Sigma_i=\partial_{01}X_i$.
There is then a decomposition of $X$ into manifolds with corners
\[
X=X_i\cup_{\partial_1X_i}\ol{\nu}\Sigma_i,\qquad\text{with}\qquad \partial X=\partial_0X_i\cup\partial_0\ol{\nu}\Sigma_i.
\]
This is depicted schematically in Figure~\ref{fig:schematic}.

\end{notation}

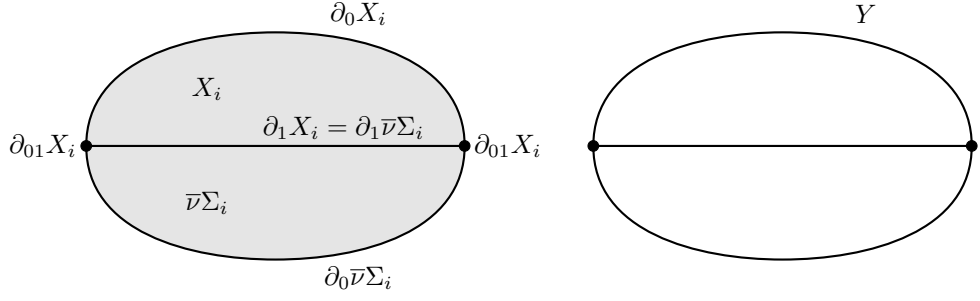
\begin{figure}
    \centering
    \begin{tikzpicture}

    \begin{scope}[]
        \draw[black, thick, fill=gray, fill opacity=0.2] (0,1.5) 
        to [out=down, in=left] (2.5, 0)
        to [out=right, in=down] (5,1.5)
        to [out=up, in=right] (2.5, 3)
        to [out=left, in=up] (0,1.5);
        \draw[black, thick] (0,1.5) -- (5,1.5);
        \filldraw[black] (0,1.5) circle (2pt) node[anchor=east]{\(\partial_{01}X_i\)};
        \filldraw[black] (5,1.5) circle (2pt) node[anchor=west]{\(\partial_{01}X_i\)};
        \node at (1.6,2.25) {\(X_i\)};
        \node at (1.6,0.75) {\(\ol{\nu}\Sigma_i\)};
        \node at (3.4,1.74) {\(\partial_1 X_i = \partial_1 \ol{\nu}\Sigma_i\)};
        \node at (3.6,3.24) {\(\partial_0 X_i\)};
        \node at (3.6,-0.24) {\(\partial_0 \ol{\nu}\Sigma_i\)};
        \end{scope}

         \begin{scope}[shift={(6.7,0)}]
        \draw[black, thick] (0,1.5) 
        to [out=down, in=left] (2.5, 0)
        to [out=right, in=down] (5,1.5)
        to [out=up, in=right] (2.5, 3)
        to [out=left, in=up] (0,1.5);
        \draw[black, thick] (0,1.5) -- (5,1.5);
        \filldraw[black] (0,1.5) circle (2pt);
        \filldraw[black] (5,1.5) circle (2pt);
        \node at (3.6,3.24) {\(Y\)};
        \end{scope}
    \end{tikzpicture}
    \caption{Left: a schematic for the decomposition of $X$ as described in \Cref{not:decomposition}. Right: the closed subspace $Y$ used in the proof of Proposition~\ref{prop:gammas}.}
    \label{fig:schematic}
\end{figure}

We now wish to analyse the Casson-Sullivan invariant of a homeomorphism $F\colon X_1\to X_2$ that restricts to a diffeomorphism on the boundary.

\begin{definition}
    Given a connected component $\Sigma_i^j\subseteq \Sigma_i$, a \emph{meridian} to that component is a simple closed curve in~$\partial_1\ol{\nu} \Sigma_i^j$ that is the boundary of a $D^2$-fibre in a tubular neighbourhood of~$\Sigma_i^j$. 
    Note that by uniqueness of tubular neighbourhoods, any two such curves are homologous in $H_1(\partial_1\ol{\nu} \Sigma_i^j;\Z/2)$.
\end{definition}

The following technical lemma shows that the Casson-Sullivan invariant of $F$ is governed by the $\Z/2$-homology classes of meridians.

\begin{proposition}\label{prop:gammas}
    Suppose $\widehat{F}\colon (X,\ol{\nu}\Sigma_1)\xrightarrow{\cong_{C^0, C^\infty}}
 (X,\ol{\nu}\Sigma_2)$ is a homeomorphism that is smooth near~$\Sigma_1$, and that moreover~$\widehat{F}$ is topologically isotopic rel.~boundary to a diffeomorphism. Then the class $\mathrm{PD}(\cs(F))$ is equal to $\sum_{j=1}^N [\gamma_j]$ in $H_1(X_1;\Z/2)$, for $\{\gamma_j\}_{j=1}^N$ meridians to some collection of pairwise distinct connected components $\Sigma^j_1$ of $\Sigma_1$. 
\end{proposition}

\begin{proof}
Throughout this proof, $\Z/2$-coefficients in homology and cohomology are understood. 
For brevity, we write for $i=1,2$
\[
\partial_i:=\partial_i\overline{\nu}\Sigma_1,\quad \partial_{01}:=\partial_{01}\overline{\nu}\Sigma_1.
\]
Define
\[Y = \partial X_1 \cup \partial(\ol{\nu}\Sigma_1)\]
and take the specific open neighbourhood
\[
\nu Y:=\nu(\partial X_1)\cup\nu(\partial(\overline{\nu}\Sigma_1));
\]
in other words, the union of an open boundary collar on $\partial X_1$ and an open tubular neighbourhood of $\partial(\overline{\nu}\Sigma_1)$. 
The closed subspace $Y$ is depicted schematically in Figure~\ref{fig:schematic}. 
As $\wh{F}$ is topologically isotopic rel.~boundary to a diffeomorphism, in particular it is a diffeomorphism upon restriction to $\partial X$. This, together with the hypothesis that $\wh{F}$ is smooth near $\Sigma_1$, implies $\wh{F}$ may be assumed to be a diffeomorphism on $\nu Y$. This means there is a Casson-Sullivan invariant~$\cs(\text{$\wh{F}$ rel.~$\nu Y$})\in H^3(X,Y)$.

We develop
the following diagram, show that it commutes, and justify the claimed isomorphisms.
\[
\begin{tikzcd}[row sep = large]
    H^2(Y,\partial X) \ar[r, "\delta"] \ar[d,"\mathrm{PALD}","\cong"'] 
    & H^3(X,Y) \ar[d,"\mathrm{PALD}","\cong"']\ar[r,"\cong"'] & 
H^3(X_1,\partial X_1)\oplus H^3(\ol{\nu}\Sigma_1,\partial \ol{\nu}\Sigma_1)
     \ar[d,"{(\mathrm{PALD},\mathrm{PALD})}","\cong"'] \\
    H_2(X\sm \partial X,X\sm Y) \ar[r, "\partial"] \ar[d, "\cong", "\alpha"']
    & H_1(X\sm Y)  \ar[d, "\cong"]
    & \ar[l,"\cong"] 
    H_1(X_1\sm\partial X_1)\oplus H_1(\ol{\nu}\Sigma_1\sm\partial)
    \ar[d,"\cong"]
     \\
      H_1(\partial_1\sm\partial_{01})\ar[r,"{(k_*,\ell_*)}"]  & H_1(X_1\sm\partial X)\oplus H_1(\ol{\nu}\Sigma_1\sm\partial_0)\ar[r,"\cong"]
      &
      H_1(X_1)\oplus H_1(\ol{\nu}\Sigma_1)
\end{tikzcd}
\]

In the top left square, the map $\delta$ is the connecting map in the long exact sequence of the triple~$(X,Y,\partial X)$ and the map $\partial $ is the connecting map in the long exact sequence of the triple~$(X\sm \partial X, X\sm Y, \emptyset)$. The downward maps denoted~$\mathrm{PALD}$ are Poincar\'{e}-Alexander-Lefschetz duality isomorphisms~\cite[Theorem~VI.8.3]{Br93}, given by cap product, applied with $X$ (or $X_1$ or $\ol{\nu}\Sigma_1$) union an open exterior collar. The top left square commutes by~\cite[Lemma VI.8.1]{Br93}; we note that Bredon only considers the case of a pair, but the proof readily extends to the case of a triple, as in our case.

For the lower left square, the right-most and lower maps are inclusion-induced. We postpone the definition of the map $\alpha$ and the proof that the square commutes to the upcoming Lemma~\ref{lem:elementary}. As the particulars are not relevant to this proof, we will proceed, assuming this map is defined and the square commutes. For future reference, the lemma applies, using
\[
X\sm\partial X=(X_1\sm\partial X)\cup (\ol{\nu}\Sigma_1\sm\partial_0)
\quad\text{and}\quad
X\sm Y=\mathrm{int}(X_1\sm\partial X)\sqcup \mathrm{int}(\ol{\nu}\Sigma_1\sm\partial_0),
\]
in the notation of that lemma.

The lower right square is all inclusion-induced and thus commutes. The inclusions are all from deformation retracts and thus all are isomorphisms.

In the top right square the horizontal arrows are inclusion-induced, and hence the square commutes by naturality of Poincar\'{e}-Alexander-Lefschetz duality.
To see that the top arrow in this square is an isomorphism, consider that the Mayer-Vietoris  sequence for 
\[
(X,Y) =\left( X_1\cup \ol{\nu}\Sigma_1, \partial X_1 \cup \partial\ol{\nu}\Sigma_1 \right)
\]
yields an isomorphism 
\[H^3(X,Y) \xrightarrow{\cong} H^3(X_1,\partial X_1) \oplus H^3(\ol{\nu}\Sigma_1,\partial\ol{\nu}\Sigma_1), \]
since the other terms in the long exact sequence are 
\[H^n(X_1 \cap \ol{\nu} \Sigma_1, \partial X_1 \cap \partial \ol{\nu}\Sigma_1) =0\]
as $X_1 \cap \ol{\nu} \Sigma_1 = \partial X_1 \cap \partial \ol{\nu}\Sigma_1$.

Having established the diagram, we now use it to prove the lemma.
The top right horizontal isomorphism in the diagram is induced by inclusions $(X_1,\partial X_1)\subseteq (X,Y)$ and $(\ol{\nu}\Sigma_1,\partial \ol{\nu}\Sigma_1)\subseteq (X,Y)$. We may now apply Lemma~\ref{lem:natCS} to each of these inclusions. More precisely, in order to arrange that these are inclusions of open submanifolds, use open neighbourhoods~$\nu(X_i)\subseteq X$ and $\nu(\overline{\nu}{\Sigma_i})\subseteq X$ as the manifolds $V_i$ in the respective applications of the lemma. We thus see that~$\cs(\text{$\wh{F}$ rel.~$\nu Y$})\in H^3(X,Y)$ is mapped to 
\[
(\cs(F),\cs(\wh{F}|_{\ol{\nu}\Sigma_1}))\in H^3(X_1,\partial X_1) \oplus H^3(\ol{\nu}\Sigma_1,\partial\ol{\nu}\Sigma_1)
\]
under this isomorphism.
Since $\wh{F}$ is smooth on $\ol{\nu}\Sigma_1$,  \cref{lem:make-homeo-smooth-on-surface-neighbourhood}, implies that $\cs(\wh{F}|_{\ol{\nu}\Sigma_1})=0$, so in fact the image of $\cs(\text{$\wh{F}$ rel.~$\nu Y$})\in H^3(X,Y)$ is
\[
(\cs(F),0) \in H^3(X_1,\partial X_1) \oplus H^3(\ol{\nu}\Sigma_1,\partial\ol{\nu}\Sigma_1).
\]

The top left horizontal map in the diagram is part of the long exact sequence of the triple $(X,Y, \partial X)$ and we consider the subsequent map $H^3(X,Y) \to H^3(X,\partial X)$ in that long exact sequence (this map is not depicted in the diagram). Another application of Lemma~\ref{lem:natCS} shows that the element $\cs(\text{$\wh{F}$ rel.~$\nu Y$})$ maps to $\cs(\wh{F})\in H^3(X,\partial X)$ under this subsequent map. As $\wh{F}$ is topologically isotopic rel.~boundary to a diffeomorphism, we have $\cs(\wh{F})=0$. Hence there exists some class $d \in H^2(Y,\partial X)$ mapping to~$\cs(\text{$\wh{F}$ rel.~$\nu Y$}) \in H^3(X,Y)$ along the top left horizontal map.

Sending $d$  clockwise around the boundary of the diagram to the bottom right corner gives $(\mathrm{PD}(\cs(F)),0)\in H_1(X_1)\oplus H_1(\ol{\nu}\Sigma_1)$. Here, we have used that one definition of Poincar\'{e}-Lefschetz duality is as the composition of the the right-most column of this diagram~\cite[\textsection~VI.9]{Br93}.
Now consider sending $d$ anti-clockwise around the diagram to the bottom right corner. Define $e := \alpha\circ \mathrm{PALD}(d) \in H_1(\partial_1\sm\partial_{01})$.
By commutativity of the diagram we have that the image of $(k_*e,\ell_*e)$ in $H_1(X_1)\oplus H_1(\ol{\nu}\Sigma_1)$ is $(\mathrm{PD}(\cs(F)),0)$. This shows that $\ell_*e=0$, and we claim this is enough to show that $e$ is a sum of meridians of~$\Sigma_1$. Given this, it follows that~$\mathrm{PD}(\cs(F)) \in H_1(X_1)$, which equals the image of $k_*e$ under $H_1(X_1 \sm \partial X) \to H_1(X_1)$, is a sum of meridians. This completes the proof, modulo the claim.  

It remains to prove the claim. For this consider the Leray--Serre spectral sequence for the fibration 
\[S^1 \to (\partial_1\sm\partial_{01}) \to (\Sigma_1\sm\partial\Sigma_1).\]
The $E^2$ page is 
\[E^2_{p,q} \cong H_p(\Sigma_1\sm\partial\Sigma_1;H_q(S^1)).\]
For $p+q=1$ this leads to  $E^{\infty}_{1,0} \cong H_1(\Sigma_1\sm\partial\Sigma_1)$ and \[E^{\infty}_{0,1} \cong \coker \big(d^2_{2,0} \colon H_2(\Sigma_1\sm\partial\Sigma_1) \to H_0(\Sigma_1\sm\partial\Sigma_1;H_1(S^1))\big),\]
where $\coker d^2_{2,0}$ is generated by meridians $[\gamma_j]$, $j=1,\dots,N$, i.e.\ boundaries of $D^2$-fibres of a tubular neighbourhood, to a sub-collection of the connected components of $\Sigma_1\sm\partial\Sigma_1$. 
We obtain a short exact sequence 
\[0 \to \coker d^2_{2,0} \to H_1(\partial_1\sm\partial_{01}) \to  H_1(\Sigma_1\sm\partial\Sigma_1) \to 0.\]
The right hand map factors as 
\[H_1(\partial_1\sm\partial_{01}) \to H_1(\ol{\nu}\Sigma_1) \xrightarrow{\cong} H_1(\Sigma_1\sm\partial\Sigma_1).\]
It follows that every element of the kernel of the map 
$H_1(\partial_1\sm\partial_{01}) \to H_1(\ol{\nu}\Sigma_1)$
is a sum of meridians $\sum_{j=1}^N [\gamma_j]$, as desired.
\end{proof}

We now prove the lemma promised in the proof of Proposition~\ref{prop:gammas}. 

\begin{lemma}
    \label{lem:elementary}
    If $M=N\cup_\partial N'$ is a union of $n$-manifolds along their boundary $\partial N=\partial N'$ then for all $r\geq 0$ there is a map $\alpha$ such that the following diagram commutes.
    \begin{equation}
    \label{eq:diagram1}
    \begin{tikzcd}
        H_r(M,\mathring{N}\sqcup\mathring{N}';\Z/2)\ar[r, "\partial"]\ar[d, "\alpha","\cong"']
        & H_{r-1}(\mathring{N}\sqcup\mathring{N}';\Z/2)\ar[d, "\cong"',"\iota"]
        \\
        H_{r-1}(\partial N;\Z/2)\ar[r, "{(k_*,\ell_*)}"]
        &
        H_{r-1}(N;\Z/2)\oplus H_{r-1}(N';\Z/2).
    \end{tikzcd}
    \end{equation}
    Here, the bottom and right maps are inclusion-induced and $\partial$ is the connecting map in the long exact sequence of the pair.
    \end{lemma}

    \begin{proof}
    Throughout this proof, $\Z/2$-coefficients in homology are understood.  To begin, we make some abstract observations. For any space $A$, the long exact sequence of the pair $(A\times[-1,1],A\times [-1,0)\sqcup (0,1])$ has the following portion.
    \begin{multline*}
       \dots\to H_r(A\times[-1,1])\xrightarrow{0} H_r(A\times[-1,1],A\times [-1,0)\sqcup (0,1])\\
       \xrightarrow{\partial} \underbrace{H_{r-1}(A\times [-1,0))\oplus H_{r-1}(A\times (0,1])}_{\cong\,{H_{r-1}(A\times \{-1\})\oplus H_{r-1}(A\times \{1\})}} 
       \xrightarrow{(i_-, i_+)} H_{r-1}(A\times[-1,1])\to\dots
    \end{multline*}
    Identifying $A=A\times\{\pm 1\}$, the image of the connecting map is the diagonal subgroup
    \[
    \Delta:=\im(\partial)=\ker(i_+, i_-)=\{(x,x)\,|\,x\in H_{r-1}(A)\}\subseteq H_{r-1}(A)\oplus H_{r-1}(A).
    \]
    Note that we have a canonical isomorphism $\Delta = H_{r-1}(A)$.  In addition, $(i_+, i_-)$ is surjective, justifying the zero map displayed in our sequence. This  shows that $\partial$ is injective. Combining all this, $\partial$ determines an isomorphism to the diagonal subgroup, which we will write as
    \[
    \beta\colon H_r(A\times[-1,1],A\times [-1,0)\sqcup (0,1])\xrightarrow{\cong} H_{r-1}(A).
    \]
    
    We turn to the proof of the lemma. 
    We consider diagram \eqref{eq:diagram2} below; its purpose is to define the isomorphism $\alpha$ and to show that the topmost triangle commutes.  The left vertical arrow is the excision isomorphism, where we excise $M \sm (\partial N\times[-1,1])\subseteq M$, the complement of a closed tubular neighbourhood of $\partial N$.  The map $(k',\ell')$ is induced by the inclusion into each component of push-offs of~$\partial N=\partial N'$ into the respective interiors.  The map $\beta$ was defined above, taking $A = \partial N$, as was the down-and-right pointing inclusion map.

     \begin{equation}
    \label{eq:diagram2}
    \begin{tikzcd}[column sep = small]
        H_r(M,\mathring{N}\sqcup\mathring{N}')\ar[rr, "\partial"] \ar[rd,"\alpha","\cong"']
        & & H_{r-1}(\mathring{N}\sqcup\mathring{N}')
        \\
        & \Delta = H_{r-1}(\partial N) \ar[rd, "\subseteq"] \ar[ur,"(k'{,}\ell')"] & \\
        H_r(\partial N\times[-1,1],\partial N\times [-1,0)\sqcup (0,1])\ar[rr, "\partial"]\ar[uu, "\cong","\mathrm{exc}"']\ar[ur, "\beta","\cong"']
        & & H_{r-1}(\partial N\times [-1,0)\sqcup (0,1])\ar[uu]
    \end{tikzcd}
    \end{equation}

   The rightmost triangle commutes by definition of $(k',\ell')$.  The abstract discussion above, again with $A=\partial N$, implies that the lower triangle commutes.  The isomorphism $\alpha$ is then defined as $\beta\circ\mathrm{exc}^{-1}$, making the leftmost triangle commute.  
    The outer square of diagram \eqref{eq:diagram2} commutes because it is induced by an inclusion of pairs. 
    The commutativity of the topmost triangle then follows from combining all of the other commutativity statements. 

    Now we use this to prove that diagram~\eqref{eq:diagram1} commutes.  For the clockwise route in~\eqref{eq:diagram1} from $H_{r-1}(\partial N)$ to $H_{r-1}(N)\oplus H_{r-1}(N')$, the topmost triangle in \eqref{eq:diagram2} yields 
    \[
    \iota \circ \partial\circ\alpha^{-1}=\iota \circ (k'_*,\ell'_*).
    \]
    Then observe that $\iota \circ (k'_*,\ell'_*) = (k_*,\ell_*)$, which is the anti-clockwise route.  So diagram~\eqref{eq:diagram1} commutes as claimed. 
    \end{proof}

\section{Killing the Casson-Sullivan invariant at the expense of stabilisation}\label{section:killing-CS}

In this section, we recall a realisation result for the Casson-Sullivan invariant derived by the first named author using unpublished work of R.~Lee. We then show how to use this to kill the Casson-Sullivan invariant at the expense of stabilisation.

\medskip

The following statement was shown in \cite{GalvinCS}*{Proposition~3.1}, making use of unpublished work of R.~Lee to show that one stabilisation suffices; see \cite{galvin_2024}. 

\begin{proposition}\label{prop:ronnielee}
There exists a homeomorphism
\[
f\colon (S^1\times S^3)\# W_1 \xrightarrow{\cong_{C^0}} (S^1\times S^3)\# W_1
\]
with $\cs(f)\neq0$, inducing $f_*=\Id_{H_2((S^1\times S^3)\# W_1)}$, and such that $f|_{\ol{\nu}(S^1\times \pt)}=\Id_{\ol{\nu}(S^1\times \pt)}$. 
\end{proposition}

We now recall the technique from~\cite{GalvinCS} for using the homeomorphism from~\cref{prop:ronnielee} to kill the Casson-Sullivan invariant of $F$, at the expense of stabilising the manifold once. 

\medskip

Suppose $\widehat{F}\colon (X,\ol{\nu}\Sigma_1)\xrightarrow{\cong_{C^0, C^\infty}}
 (X,\ol{\nu}\Sigma_2)$ is a homeomorphism that is smooth near~$\Sigma_1$.
Let~$\mu\subseteq \mathring{X}_1$ be an embedded loop in the interior of $X_1$. We consider the circle sum operation of $X_1$ and~$X_2$ with $(S^1\times S^3)\# W_1$, along $\mu$ and $F(\mu)$ respectively, to obtain
\[
X_1 \#_{\mu= S^1\times\mathrm{pt}} \big((S^1 \times S^3) \# W_1\big) \cong X_1 \# W_1
\] 
and
\[
X_2 \#_{F(\mu)= S^1\times\mathrm{pt}} \big((S^1 \times S^3) \# W_1\big) \cong X_2 \# W_1.
\] 
Recall that $F\colon X_1\to X_2$ is the restriction of $\wh{F}$ to the surface exteriors. As the homeomorphism
\[
f \colon (S^1 \times S^3) \# W_1 \xrightarrow{\cong_{C^0}} (S^1 \times S^3) \# W_1,
\]
from \cref{prop:ronnielee} fixes a neighbourhood of $S^1\times\mathrm{pt}$ pointwise, it makes sense to extend the homeomorphism $F$ over the circle-summed manifolds to obtain a homeomorphism
 \[
    G':= F\#_{\mu= S^1\times\mathrm{pt}}f\colon X_1\# W_1 \xrightarrow{\cong_{C^0}}  X_2\# W_1. 
    \]

\begin{theorem}\label{thm:galvin-Thm-32}
 Suppose $\widehat{F}\colon (X,\ol{\nu}\Sigma_1)\xrightarrow{\cong_{C^0, C^\infty}}
 (X,\ol{\nu}\Sigma_2)$ is a homeomorphism that is smooth near~$\Sigma_1$, and that $\widehat{F}$ is topologically isotopic rel.~boundary to a diffeomorphism. Let $\mu\subseteq \mathring{X}_1$ be an embedded loop homologous to $\PD(\cs(F))$.  Then the circle-summed homeomorphism along~$\mu$ 
    \[
    G':= F\#_{\mu= S^1\times\mathrm{pt}}f\colon X_1\# W_1 \xrightarrow{\cong_{C^0}}  X_2\# W_1. 
    \]
    has trivial Casson-Sullivan invariant.
\end{theorem}

\begin{proof}
    The result follows from a direct application of {\cite{GalvinCS}*{Theorem~3.2}}.
\end{proof}

\begin{corollary}\label{cor:G-prime-top-stably-PI-to-a-diffeo}
After some number $k-1$ of further stabilisations, $G' \# \Id_{W_{k-1}}$ is topologically pseudo-isotopic to a diffeomorphism, denoted  
 \[
G \colon X_1 \# W_k \xrightarrow{\cong_{C^{\infty}}} X_2 \# W_k.
 \]
\end{corollary}

\begin{proof}
Combine \cref{prop:key-property-of-CS} and \cref{thm:galvin-Thm-32}.
\end{proof}

\section{Obtaining a diffeomorphism of \texorpdfstring{$X\# W_k$}{X stabilised} sending \texorpdfstring{$\Sigma_1$}{one surface} to \texorpdfstring{$\Sigma_2$}{the other} that is topologically pseudo-isotopic to the identity }\label{section:obtaining-diffeo-top-PI-to-Id}

Suppose $\widehat{F}\colon (X,\ol{\nu}\Sigma_1)\xrightarrow{\cong_{C^0, C^\infty}}
 (X,\ol{\nu}\Sigma_2)$ is a homeomorphism that is smooth near~$\Sigma_1$, that is also topologically isotopic rel.~boundary to the identity. The outcome of the previous section is that, after a single stabilisation of~$X_1$ and~$X_2$, there is a way to extend $F\colon X_1\to X_2$, the restriction of~$\wh{F}$, across the added $W_1 = S^2\times S^2$, in such a way that the resulting homeomorphism $G'\colon X_1\# W_1 \to X_1\# W_1$ has trivial Casson-Sullivan invariant. This means we may stabilise~$G'$ and obtain a homeomorphism topologically isotopic, rel.~boundary, to a diffeomorphism $G$ between the stabilised exteriors. This can then be filled back in to give a diffeomorphism of stabilised $X$. In this section, we will argue that, because the original $\widehat{F}$ was topologically isotopic rel.~boundary to the identity, the eventual diffeomorphism of the stabilised $X$ can be chosen so as to be topologically pseudo-isotopic to the identity.
 
\medskip

We will need the following lemma in this section.

\begin{lemma}\label{lem:nice-P-subset-means-isotopic-homeos}
    If $g,h\colon M \xrightarrow{\cong_{C^0}} M$ are two orientation-preserving homeomorphisms of a compact $4$-manifold $M$ such that $g=h$ on $M\sm \mathring{P}$, where $P$ is a compact, simply-connected codimension~$0$ submanifold $P\subseteq \mathring{M}$, with $\partial P$ connected, $\dim H_1(\partial P;\Q)\leq 1$, and $g_*=h_*\colon H_2(P)\to H_2(P)$, then $g$ is topologically isotopic rel.\ boundary to $h$.
\end{lemma}

\begin{proof}
    Apply~\cite{Orson-Powell}*{Corollary~C} to conclude that $g|_P$ and $h|_P$ are topologically isotopic rel.~boundary. Extend this by the identity isotopy to the rest of $M$ to conclude that $g$ is topologically isotopic rel.~boundary to~$h$.
\end{proof}

We describe a standard way to ``fill'' a homeomorphism between the surface-exteriors back to a homeomorphism of the whole manifold.

\begin{definition}\label{remark:convention-on-hats}
    Suppose $\widehat{F}\colon (X,\ol{\nu}\Sigma_1)\xrightarrow{\cong_{C^0, C^\infty}}
 (X,\ol{\nu}\Sigma_2)$ is a homeomorphism that is smooth near~$\Sigma_1$.
    Let $n\geq 0$. For any homeomorphism $G\colon X_1\# W_n\to X_2\# W_n$ such that $G=\wh{F}$ on~$\partial_1 X_1$, define the extension \[\wh{G}:= G\cup \wh{F}\vert_{\nu \Sigma_1}\colon X\# W_n \to X\# W_n.\]  If $G$ is a diffeomorphism then $\wh{G}$ is also a diffeomorphism. Note $\wh{G}$ sends $\Sigma_1$ to~$\Sigma_2$. 
\end{definition}

\begin{lemma}\label{lem:wh-G-top-PI-to-Id}
Suppose $\widehat{F}\colon (X,\ol{\nu}\Sigma_1)\xrightarrow{\cong_{C^0, C^\infty}}
 (X,\ol{\nu}\Sigma_2)$ is a homeomorphism that is smooth near~$\Sigma_1$, and that $\widehat{F}$ is topologically isotopic rel.~boundary to $\Id_X$. Then for some choice of embedded loop~$\mu\subseteq \mathring{X}_1$
 homologous to $\PD(\cs(F))$, the resultant $G'$ from~\cref{thm:galvin-Thm-32} and corresponding
    \[
    G \colon X_1 \# W_k \xrightarrow{\cong_{C^{\infty}}} X_2 \# W_k
    \]
    from~\cref{cor:G-prime-top-stably-PI-to-a-diffeo} are such that the diffeomorphism $\wh{G} \colon X\# W_k \xrasmooth X\# W_k$ is topologically pseudo-isotopic to $\Id_{X\# W_k}$.
\end{lemma}

\begin{remark}
It is worth recalling again that our convention is isotopies are explicitly specified to be rel.\ boundary when relevant, but that all pseudo-isotopies are by definition rel.\ boundary, so we do not need to specify this property each time.
\end{remark}

\begin{proof}
By definition, $\wh{F}$, and hence also suitable stabilisations thereof, is topologically isotopic rel.\ boundary to the identity.  So $\Id_{X\# W_k}$ and $\wh{F} \# \Id_{W_k}$ are topologically isotopic rel.\ boundary. 

We now choose $\mu$. By~\cref{prop:gammas}, we may represent the homology class $\PD(\cs(F))$ by the sum $\sum_{j=1}^N[\gamma_j] \in H_1(X_1;\Z/2)$, where
$\{\gamma_j\}_{j=1}^N$ is a collection of meridians to $\Sigma$. 
For each $j\in\{1,\dots,{N-1}\}$, choose a smoothly embedded path $\alpha_{j,j+1}\subseteq X_1$ from $\gamma_j$ to $\gamma_{j+1}$. Perform band sums using this arc collection, and push this curve slightly off $\partial X_1$, to yield the desired~$\mu$. Observe that the map $G'$ from~\cref{thm:galvin-Thm-32} and $\wh{F}\# \Id_{W_1}$ agree upon restriction to
\[
\big(X_1\# W_1\big)\sm(\nu(\mu)\# W_1).
\]
In particular $G'=\wh{F}$ on $\partial_1X_1$, so we may extend $G'$ over the tubular neighbourhood using~$\wh{F}$, to define $\widehat{G}'$.

Now, for each $j\in\{1,\dots,{N-1}\}$, let $d_j\subseteq X$ be a meridional disc to the connected component~$\Sigma^j_1$ with boundary the meridian $\gamma_j$. We define
\[
V := \nu\Big(\bigcup_{j=1}^N  d_j\cup \bigcup_{j=1}^{N-1} \alpha_{j,j+1}\Big) \# W_1.
\]
\cref{lem:nice-P-subset-means-isotopic-homeos} may now be applied to the maps $\wh{G}'$ and $\wh{F}\#\Id_{W_1}$, with $M=X\#W_1$ and $P=V$ in the language of that lemma. To see this, note that since~$\nu\mu\subseteq V$, we have $\wh{G}'=\wh{F}\#\Id_{W_1}$ on $X\sm V$. Moreover, $\pi_1(V) = \{1\}$, $\partial P\cong S^3$ and $(\wh{G}')_*=\Id=(\wh{F}\#\Id_{W_1})_*$ on~$H_2(P)$. 
Thus \cref{lem:nice-P-subset-means-isotopic-homeos} implies that $\wh{F} \# \Id_{W_1}$ and $\wh{G}'$ are topologically isotopic rel.\ boundary.
Thus $\wh{F} \# \Id_{W_k}$ and $\wh{G}' \# \Id_{W_{k-1}}$ are topologically isotopic rel.\ boundary. 

Finally since $G' \# \Id_{W_{k-1}}$ and $G$  are topologically pseudo-isotopic by \cref{cor:G-prime-top-stably-PI-to-a-diffeo}, it follows that after filling in we have that 
$\wh{G}' \# \Id_{W_{k-1}}$ and $\wh{G}$ are topologically pseudo-isotopic. 

We consider the two isotopies produced above as pseudo-isotopies, and 
concatenate all three pseudo-isotopies to obtain the desired pseudo-isotopy between $\wh{G}$ and~$\Id_{X\# W_k}$. 
\end{proof}

In the next lemma we make use of the topological Hatcher--Wagoner obstruction $\Sigma^{\Top}(\Psi) \in \Wh_2(\pi)$ of a topological pseudo-isotopy $\Psi \colon X \times I \to X \times I$, developed by Nonino and the first-named author~\cite{Galvin:2025aa}. 

\begin{lemma}\label{lem:sigma-comes-from-exterior}
Let 
\[
\Psi \colon (X \# W_k) \times I \xrightarrow{\cong_{C^{0}}} (X \# W_k) \times I
\]
be a pseudo-isotopy as produced by \cref{lem:wh-G-top-PI-to-Id}. 
    Then the primary Hatcher--Wagoner obstruction~$\Sigma^{\Top}(\Psi)$ lies in the image of
    \[
    \Wh_2(\pi_1(X_1)) \to \Wh_2(\pi).
    \]
\end{lemma}

\begin{proof}
    This follows from the construction of $\Psi$. The first two pseudo-isotopies in the construction come from topological isotopies, so these have vanishing obstructions in $\Wh_2(\pi)$. The final pseudo-isotopy is obtained from a pseudo-isotopy of $X_1 \# W_k$ by filling in the tubular neighbourhoods of the $\Sigma_i$, and hence by naturality of $\Sigma^{\Top}$ \cite{Galvin:2025aa}*{Proposition 1.3} we see that $\Sigma^{\Top}(\Psi)$ lies in the image of $\Wh_2(\pi_1(X_1)) \to \Wh_2(\pi)$, as desired. 
\end{proof}

\section{Improving \texorpdfstring{$\Psi$}{the topological pseudo-isotopy} to a smooth pseudo-isotopy after further stabilisations}\label{section:improving-top-PI-to-smooth-PI}

The next step will be to improve our topological pseudo-isotopy $\Psi$ into a smooth pseudo-isotopy, possibly after further stabilising $X \# W_k$ to $X \# W_m$ for some $m \geq k$.    
Before splitting into cases, we collect some results that will be useful during the proofs.

\subsection{The smoothing obstruction for topological pseudo-isotopies}
Given a smooth, compact 4-manifold $U$ and a topological pseudo-isotopy $F \colon U \times I \to U \times I$ that is smooth near~$\partial(U \times I)$, Orson--Powell--Randal-Williams defined an obstruction $\KS(F) \in H_2(U;\Z/2)$, as follows.  Let $\sigma$ denote the product smooth structure on $U \times I$. We consider $U \times I \times I$ as a topological manifold and place a smooth structure on its boundary.  Endow $U \times I \times \{1\}$ with the pullback smooth structure~$F^*\sigma$, and then use the standard structure on the rest of $\partial(U \times I \times I)$. Denote the resulting smooth manifold by $(U \times I \times I)_F$. We then consider the Kirby-Siebenmann obstruction to extending this smooth structure over $X \times I \times I$:
\[\ks(U \times I \times I,\partial(U \times I \times I)_F) \in H^4(X \times I \times I,\partial;\Z/2).\]
Taking its Poincar\'{e} dual, and implicitly applying that $U \simeq U \times I \times I$, we obtain: 
\[\KS(F) := PD\big(\ks(U \times I \times I,\partial(U \times I \times I)_F)\big) \in  H_2(U;\Z/2).\]

The next proposition gives some useful properties of $\KS$, that it is a homomorphism, it detects the difference between smooth and topological pseudo-isotopies, and that it satisfies a gluing formula. 

\begin{proposition}\label{lem:properties-of-KS}
Let $Q(U)$ denote the group of topological pseudo-isotopies $F \colon U \times I \to U \times I$ such that $F_1:= F|_{U \times \{1\}}$ is smooth, up to topological pseudo-isotopy relative a smooth pseudo-isotopy on $U \times I \times \{1\}$ and to the identity on the rest of $\partial(U \times I \times I)$. The group structure comes from composition.
    \begin{enumerate}[label=(\roman*)]
                \item\label{item:lem-properties-KS-i} The map $\KS$ is a homomorphism $\KS \colon Q(U) \to H_2(U;\Z/2)$.
        \item\label{item:lem-properties-KS-ii} We have that $\KS(F)=0$ if and only if $F$ is topologically isotopic rel.\ $\partial(U \times I)$ to a smooth pseudo-isotopy. 
        \item\label{item:lem-properties-KS-iii} 
        Suppose $M$ and $N$ are manifolds with corners, where $\partial M=\partial_0M\cup\partial_1M$, $\partial N=\partial_0N\cup\partial_1N$ and $\partial_1M=\partial_1N$. Let $U=M\cup_{\partial_1M=\partial_1N}N$.
        Let  $\wh{\Psi} \colon U \times I \xrightarrow{\cong_{C^0}} U \times I$ be a topological pseudo-isotopy that restricts to a smooth pseudo-isotopy $\Psi_M \colon M \times I \xrasmooth M \times I$ and to a topological pseudo-isotopy $\Psi_{N} \colon N \times I \xratop N\times I$. Then under the inclusion induced map 
        \begin{align*}
           (i_N)_*\colon H_2(N;\Z/2) &\to H_2(U;\Z/2) \\
            \KS(\Psi_{N}) &\mapsto \KS(\wh{\Psi}). 
        \end{align*} 
    \end{enumerate}
\end{proposition}

\begin{proof}
 Part \ref{item:lem-properties-KS-i} is proven in \cite{OPRW}*{Lemma~3.2}, and part \ref{item:lem-properties-KS-ii} is proven in \cite{OPRW}*{Theorem~A}. 
 We prove \ref{item:lem-properties-KS-iii}. Let 
 \begin{align*}
     A  := M \times I \times I,\;\;\; 
     B := N  \times I \times I,\;\;\;
     C := \partial A,\;\;\;
     D := \partial B,\text{ and }
     E  := C \cap D =  A \cap B.
 \end{align*}
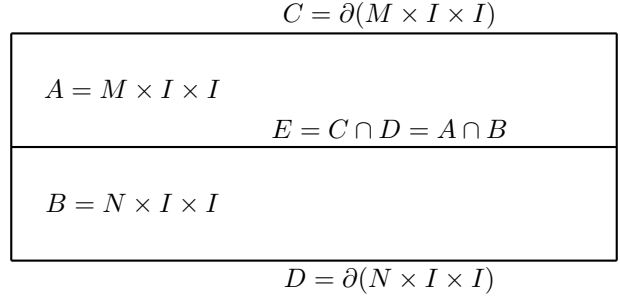
\begin{figure}
    \centering
    \begin{tikzpicture}
        \draw[black, thick] (0,0) -- (8,0);
        \draw[black, thick] (0,3) -- (8,3);
        \draw[black, thick] (0,3) -- (0,0);
        \draw[black, thick] (0,1.5) -- (8,1.5);
        \draw[black, thick] (8,3) -- (8,0);
        \node at (1.6,2.25) {\(A=M\times I\times I\)};
        \node at (1.6,0.75) {\(B=N\times I\times I\)};
        \node at (5,1.74) {\(E=C\cap D=A\cap B\)};
        \node at (5,3.24) {\(C=\partial(M\times I\times I)\)};
        \node at (5,-0.26) {\(D=\partial(N\times I\times I)\)};
    \end{tikzpicture}
    \caption{A schematic for the decomposition used in the proof of \Cref{lem:properties-of-KS}.}
    \label{fig:schematic2}
\end{figure}
Endow $C\cup D$ with the standard smooth structure on $(U\times I\times\{0\})\cup((\partial M\cup\partial N)\times I\times[0,1))$ and on $U\times I\times\{1\}$ pull back the smooth structure using $\wh{\Psi}$. As the restriction $\wh{\Psi}|_{(\partial M\cup \partial N)\times I}$ is a diffeomorphism, this indeed determines a smooth structure on $C\cup D$. Consider the diagram
\[
\begin{tikzcd}
    H^4(A,C;\Z/2)\oplus H^4(B,D;\Z/2)\ar[d, "PD\oplus PD", "\cong"']
    &
    \ar[l, "\cong"'] H^4(A\cup B, C\cup D;\Z/2)\ar[r]
    &
    H^4(A\cup B,\partial(A\cup B);\Z/2)\ar[d,"PD", "\cong"']
    \\
    H_2(M;\Z/2)\oplus H_2(N;\Z/2)\ar[rr]
    &&
    H_2(U;\Z/2)
\end{tikzcd}
\]
where the top left map comes from the Mayer-Vietoris sequence
\[0= H^3(E,E) \to H^4(A\cup B, C\cup D;\Z/2) \xrightarrow{\cong }H^4(A,C;\Z/2)\oplus H^4(B,D;\Z/2) \to H^4(E,E)=0,\] 
noting that $E= A \cap B = C \cap D$. 
The top right map is induced by inclusion.  The diagram commutes by naturality and linearity of cap products.  As $C$ and $D$ have smooth structures, there are Kirby-Siebenmann invariants $\ks(A,C) \in H^4(A,C;\Z/2)$ and $\ks(B,D) \in H^4(B,D;\Z/2)$. By naturality of the Kirby-Siebenmann invariant, we have that \[(\ks(A,C),\ks(B,D))\mapsto \ks(A\cup B,C\cup D)\mapsto \ks(A\cup B,\partial(A\cup B))\] 
along the top row of the diagram. The clockwise composition around the diagram thus sends $(\ks(A,C),\ks(B,D))$ to~$\KS(\wh{\Psi})$.  The Kirby-Siebenmann invariant $\ks(A,C)$ vanishes because the restriction of $\wh{\Psi}$ to $M\times I$ is a diffeomorphism, so the standard smooth structure on $M\times I\times I$ extends that on the boundary~$C$. Thus the anticlockwise composition is $(\ks(A,C),\ks(B,D))=(0,\ks(B,D))\mapsto (0,\KS(\Psi_N))\mapsto 0+(i_N)_*\KS(\Psi_N)$. Thus $(i_N)_*\KS(\Psi_N)=\KS(\wh{\Psi})$ as claimed.
\end{proof}

We will also need realisation results for the smoothing obstruction $\KS$. The first one allows us to realise all possible values, at the expense of stabilising the $4$-manifold (which in the context of this article is no problem).  

\begin{theorem}\label{prop:realisation-of-OPRW-KS-obstruction} 
    There exists $m \geq k$ such that for every $x \in H_2(X_1\# W_m;\Z/2)$ there is a topological pseudo-isotopy
    \[
    \Phi \colon (X_1 \# W_m) \times I \xrightarrow{\cong_{C^0}} (X_1 \# W_m) \times I\]
    such that $\Phi|_{\sqsubset} = \Id_{\sqsubset}$, the restriction 
    \[
    \Phi|_{(X_1 \# W_m) \times \{1\}} \colon X_1\# W_m \times \{1\} \xrightarrow{\cong_{C^{\infty}}} X_1 \# W_m  \times \{1\}
    \]
    is a diffeomorphism, and $\KS(\Phi) =x \in H_2(X_1\# W_m;\Z/2)$.
\end{theorem}

\begin{proof}
    This is an application of Orson--Powell--Randal-Williams~\cite{OPRW}*{Theorem~E}.
\end{proof}

In \cref{prop:realisation-of-OPRW-KS-obstruction} we obtain a topological pseudo-isotopy from the identity to some diffeomorphism, but we have no control on the diffeomorphism that arises on $X_1\# W_m \times \{1\}$.  In the next result, also proven in Orson--Powell--Randal-Williams, we obtain more control on this diffeomorphism, at the expense of requiring more assumptions.

\begin{theorem}\label{thm:inertial-realisation-result-OPRW}
Let $U$ be a compact, smooth, orientable 4-manifold with $\pi := \pi_1(U)$. 
Suppose that the map $I_2 \colon H_2(\pi;\Z_{(2)}) \to L_6(\Z\pi)_{(2)}$ is zero, and there is no 2-torsion in $H_1(\pi;\Z)$. Then for every $x \in H_2(U;\Z/2)$ there is an topological pseudo-isotopy $\Phi \colon U \times I \xratop U \times I$ with $\KS(\Phi) =x$ and $\Phi|_{\partial(U \times I)} = \Id_{U \times I}$, i.e.\ $\Phi$ is an inertial pseudo-isotopy. 
\end{theorem}

\begin{proof}
   This follows directly from the proof of \cite{OPRW}*{Theorem~C}, where it is shown that the composition $\pi_1(\pseudoHomeo^+_{\partial}(U)) \to Q(U) \xrightarrow{\KS} H_2(U;\Z/2)$ is surjective; see the start of \cite{OPRW}*{Section~5}. 
\end{proof}

\cref{prop:realisation-of-OPRW-KS-obstruction} will be used for case \ref{thm:main-item-iii} of \cref{thm:main}, while \cref{thm:inertial-realisation-result-OPRW} will be needed for case~\ref{thm:main-item-iv}.

\subsection{Smoothing the topological pseudo-isotopy when the obstruction is supported on the surface exterior}

The proof of Theorem~\ref{thm:main} in case \ref{thm:main-item-iii} will make use of the following key result.

\begin{proposition}\label{prop:ks-comes-from-exterior}
    Let  $\widehat{G}\colon X\xrightarrow{\cong_{C^\infty}} X$ be a diffeomorphism sending $\Sigma_1$ to $\Sigma_2$ and write $G$ for the restriction of $\wh{G}$ to $X_1=\overline{\nu}\Sigma_1$.
    Let 
\[
\Psi \colon (X \# W_k) \times I \xrightarrow{\cong_{C^{0}}} (X \# W_k) \times I
\]
be a topological pseudo-isotopy from $\wh{G}$ to the identity map $\Id_{X\#W_k}$.
Assume that $\KS(\Psi)\in \im(H_2(X_1;\Z/2)\to H_2(X;\Z/2))$. 
Then there exists~$m\geq k$ and a diffeomorphism 
    \[
    L\colon X_1 \# W_m \xrightarrow{\cong_{C^{\infty}}} X_2\# W_m
    \]
    with $\wh{L} \colon X \# W_m \to X \# W_m$ smoothly pseudo-isotopic to the identity, via a smooth pseudo-isotopy 
    \[\Xi\colon (X\# W_m) \times I\xrightarrow{\cong_{C^{\infty}}} (X\# W_m)\times I\] with the property that $\Sigma(\Xi)\in \im(\Wh_2(\pi_1(X_1))\to \Wh_2(\pi))$.
\end{proposition}

\begin{proof}
    Write $\KS(\Psi)=x\in H_2(X;\Z/2)$.  By assumption, we may choose a lift $\ol{x}\in H_2(X_1;\Z/2)$ of~$x$.
    Now apply \cref{prop:realisation-of-OPRW-KS-obstruction} to $\ol{x}$, to obtain, for some $m \geq k$, a topological pseudo-isotopy 
    \[
    \Phi\colon (X_1\# W_m)\times I\xrightarrow{\cong_{C^{0}}} (X_1\# W_m)\times I
    \]
    with $\KS(\Phi)=\ol{x} \in H_2(X_1;\Z/2)$, that restricts on $(X_1\# W_m)\times \{1\}$ to a diffeomorphism. Extend~$\Phi$ to the whole of $X$ using the trivial pseudo-isotopy on $\ol{\nu}\Sigma_1 \times I$, to obtain 
    \[
    \wh{\Phi} \colon (X\#W_m) \times I \xrightarrow{\cong_{C^{0}}} (X\#W_m) \times I.
    \] 
We therefore have that 
\[x = i_*(\ol{x}) = i_*(\KS(\Phi)) = \KS(\wh{\Phi}),\]
where $i \colon X_1 \to X$ is the inclusion, and the last equality follows from \cref{lem:properties-of-KS}~\ref{item:lem-properties-KS-iii}.  
    
    Compose $\wh{\Phi}$ with $\Psi$ (stabilised by extending by the identity on $W_{m-k} \times I$, but still denoted by $\Psi$) to form a topological pseudo-isotopy
    \[\Psi \circ \wh{\Phi} \colon (X\# W_m) \times I\xrightarrow{\cong_{C^{\infty}}} (X\# W_m)\times I.\]
The restriction of this to $(X_1\# W_m) \times \{1\}$ yields $L := (G \# \Id_{W_{m-k}}) \circ \Phi_1 \colon X_1 \# W_m \xrightarrow{\cong_{C^{\infty}}} X_2\# W_m$, and 
the restriction to $(X\# W_m) \times \{1\}$ yields 
\[\wh{L} := (\wh{G} \# \Id_{W_{m-k}}) \circ \Phi_1 \colon X \# W_m \xrightarrow{\cong_{C^{\infty}}} X\# W_m.\]
    
    Since  $\KS$ is a homomorphism \cref{lem:properties-of-KS}~\ref{item:lem-properties-KS-i}, we have $\KS(\Psi \circ \wh{\Phi}) = \KS(\Psi) + \KS(\wh{\Phi})=x+x=0$. Hence by \cref{lem:properties-of-KS}~\ref{item:lem-properties-KS-ii}, we have that $\Psi \circ \wh{\Phi}$ is topologically isotopic rel.\ boundary to a smooth pseudo-isotopy, as desired, which we denote by $\Xi$.

It remains to prove the statement regarding $\Sigma(\Xi)$.  We calculate
    \[
    \Sigma(\Xi)=\Sigma^{\Top}(\Xi)=\Sigma^{\Top}(\Psi\circ \widehat{\Phi})=\Sigma^{\Top}(\Psi)+\Sigma^{\Top}(\widehat{\Phi})\in\im(\Wh_2(\pi_1(X_1))\to \Wh_2(\pi)),
    \]
    where the first equality follows from \cite{Galvin:2025aa}*{Theorem 1.1}, the second from the definition of~$\Xi$, and the third from the fact that $\Sigma^{\Top}$ is a homomorphism~\cite{Galvin:2025aa}*{Lemma 3.12}.  That this sum is an element of the noted subgroup can be deduced as follows.  The first summand is by \cref{lem:sigma-comes-from-exterior} and the second summand is since $\wh{\Phi}$ was only supported on the exterior $(X_1\# W_m)\times I$, and since~$\Sigma^{\Top}$ is natural under inclusions of codimension zero submanifolds~\cite{Galvin:2025aa}*{Proposition 1.3}.
\end{proof}

\subsection{Smoothing the pseudo-isotopy assuming at least one of the hypotheses of \texorpdfstring{\cref{thm:main}}{Theorem A} hold}

Next we prove that in both sets of assumptions of \cref{thm:main}, we can improve~$\Psi$, after further stabilising, to a smooth pseudo-isotopy, while preserving its other useful properties. 
The aim is to prove the following result. The output is similar to \cref{prop:ks-comes-from-exterior}, but instead of assuming that $\KS(\Psi)\in \im(H_2(X_1;\Z/2)\to H_2(X;\Z/2))$, we assume one of the hypotheses of \cref{thm:main}.

\begin{proposition}\label{prop:aim-of-smoothing-PI-section}
Suppose that for some $k \geq 0$ there exists a diffeomorphism $\wh{G} \colon X \# W_k \xrasmooth X \# W_k$ that sends $\Sigma_1$ to $\Sigma_2$ and that is topologically pseudo-isotopic to the identity, via a topological pseudo-isotopy $\Psi$ such that $\Sigma^{\Top}(\Psi)\in \im(\Wh_2(\pi_1(X_1))\to \Wh_2(\pi))$.  
Suppose that at least one of the conditions 
\ref{thm:main-item-iii} or \ref{thm:main-item-iv} in \cref{thm:main} is satisfied. 

Then there exists~$m\geq k$ and a diffeomorphism $L\colon X_1 \# W_m \xrightarrow{\cong_{C^{\infty}}} X_2\# W_m$, which extends to a diffeomorphism \[\wh{L} \colon X \# W_m \xrasmooth X \# W_m\] that sends $\Sigma_1$ to $\Sigma_2$, and is smoothly pseudo-isotopic to the identity, via a smooth pseudo-isotopy 
    \[\Xi\colon (X\# W_m) \times I\xrightarrow{\cong_{C^{\infty}}} (X\# W_m)\times I\] such that $\Sigma(\Xi)\in \im(\Wh_2(\pi_1(X_1))\to \Wh_2(\pi))$.
\end{proposition}

The proof of \cref{prop:aim-of-smoothing-PI-section} in case~\ref{thm:main-item-iii} reduces to showing that the assumption of \cref{prop:ks-comes-from-exterior} can be arranged to hold, and then applying \cref{prop:ks-comes-from-exterior}.  The proof for case \ref{thm:main-item-iv} uses \cref{thm:inertial-realisation-result-OPRW} instead.  

\subsubsection{Proof of \cref{prop:aim-of-smoothing-PI-section} in case \ref{thm:main-item-iii}}

In this scenario we assume that for each connected component $\Sigma_1^j$ of $\Sigma_1$, we have that ~$\lambda^{\Z/2}(x,[\Sigma_1^j]) = 0 \in \Z/2$ for all $x\in H_2(X;\Z/2)$.

We consider the inclusion $i \colon X_1 \to X$.  
We want to show that the hypothesis of \cref{prop:ks-comes-from-exterior}, that \[i_* \colon \KS(\Psi)\in \im(i_* \colon H_2(X_1;\Z/2)\to H_2(X;\Z/2)),\]  
is satisfied. Then \cref{prop:ks-comes-from-exterior} directly implies \cref{prop:aim-of-smoothing-PI-section}. 

To show this, we first compute that 
\begin{align*}
H_2(X,X_1;\Z/2) &\xrightarrow{\cong} H_2(\ol{\nu} \Sigma_1,\partial_1 \ol{\nu}\Sigma_1;\Z/2) \cong H^2(\ol{\nu}\Sigma_1, \partial_0 \ol{\nu}\Sigma_1;\Z/2) \cong H^2(\Sigma_1,\partial \Sigma_1;\Z/2) \\ &\cong H_0(\Sigma_1;\Z/2) \cong (\Z/2)^c,
    \end{align*}
where $c$ is the number of connected components $\Sigma_1^j$ of $\Sigma_1$.
Consider the exact sequence of the pair:
\[H_2(X_1;\Z/2) \xrightarrow{i_*} H_2(X;\Z/2) \to H_2(X,X_1;\Z/2) \cong (\Z/2)^c.\]
The latter group is generated by meridional discs to the connected components of $\Sigma_1$. Then $\coker(i_*) \subseteq (\Z/2)^c$ is generated by a collection of homology classes $\{S_\ell\} \subseteq H_2(X;\Z/2)$ such that for each $S_{\ell}$ the $\Z/2$-intersection pairing 
\[\lambda^{\Z/2} \colon H_2(X;\Z/2) \times H_2(X,\partial X;\Z/2)\to\Z/2\] satisfies that $\lambda^{\Z/2}(S_{\ell},[\Sigma_1^j]) = 1 \in \Z/2$ for some $j$.  If $\lambda^{\Z/2}(x,[\Sigma_1^j]) = 0 \in \Z/2$ for all $x\in H_2(X;\Z/2)$, then there is no such collection~$\{S_{\ell}\}$.  It follows that $\coker(i_*)=0$, so $i_*$ is surjective, and hence $\KS(\Psi)\in \im(i_* \colon H_2(X_1;\Z/2)\to H_2(X;\Z/2))$ as desired.  \qed

\subsubsection{Proof of \cref{prop:aim-of-smoothing-PI-section} in case \ref{thm:main-item-iv}}

For  case \ref{thm:main-item-iv}, we assume that the component  \[I_2 \colon H_2(\pi;\Z_{(2)}) \to L_6(\Z\pi)_{(2)}\] of the algebraic assembly map  is zero, there is no 2-torsion in $H_1(\pi;\Z)$, and that $\Wh_2(\pi)=0$. 

Under the first two of these hypotheses, by \cref{thm:inertial-realisation-result-OPRW}, there exists a topological pseudo-isotopy from the identity to itself 
\[\Phi\colon (X\# W_k)\times I\xratop (X\# W_k)\times I,\]  such that $\KS(\Psi)=\KS(\Phi)$. Then since $\KS$ is a homomorphism by \cref{lem:properties-of-KS}~\ref{item:lem-properties-KS-i}, we have that $\KS(\Phi\circ\Psi)=0$ so that $\Phi\circ\Psi$ is topologically isotopic rel.~boundary to a smooth pseudo-isotopy $\Xi$ from $\Xi|_{(X\# W_k)\times\{1\}}=\Psi|_{(X\# W_k)\times\{1\}}$ to $\Id_{X\# W_k}$. Note that $\wh{L}:= \Xi|_{(X\# W_k)\times\{1\}}$ sends~$\Sigma_1$ to $\Sigma_2$.

Also since $\Wh_2(\pi)=0$, it is automatic  that $\Sigma(\Xi) \in \im(\Wh_2(\pi_1(X_1)) \to \Wh_2(\pi)=0)$. 
Hence the smooth pseudo-isotopy $\Xi$ satisfies the conclusion of \cref{prop:aim-of-smoothing-PI-section}. \qed

\section{Completing the proof of \texorpdfstring{\cref{thm:main}}{Theorem A}}\label{section:completing-the-proof}

The following results of Singh and Gabai, respectively, will be essentially used in the conclusion of the proof.

\begin{theorem}[{\cite{Singh}*{Theorem~E}}]
\label{thm:singh}
    Let $M$ be a smooth, compact $4$-manifold and $x\in\Wh_2(\pi_1(M))$. Then there exists $N\in\N$ and a smooth pseudo-isotopy
    \[
    \Lambda\colon (M\#W_N)\times I\to (M\#W_N)\times I,
    \]
    restricting to the identity on $(M\#W_N)\times\{0\}$ and $(\partial M\#W_N)\times I$,
    such that
    \[
    \Sigma(\Lambda)=x\in\Wh_2(\pi_1(M\#W_N))=\Wh_2(\pi_1(M)).
    \]
\end{theorem}

\begin{theorem}[\cite{Gabai-22}*{Theorem~2.5}]\label{thm:PI-iff-stable-smooth-isotopy}
    Let $f\colon M\xrightarrow{\cong_{C^\infty}} M$ be an orientation-preserving diffeomorphism of a smooth, compact, oriented $4$-manifold $M$. Then $f$ is smoothly stably isotopic rel.~boundary to $\Id_M$ if and only if $f$ is smoothly pseudo-isotopic to $\Id_M$ via a smooth pseudo-isotopy $\Xi$ with vanishing Hatcher--Wagoner obstruction $\Sigma(\Xi)=0\in \mathrm{Wh}_2(\pi_1(M))$.
\end{theorem}

We apply the results of Singh and Gabai in the following lemma, then proceed to complete the proof of the main theorem.

\begin{lemma}\label{lemma:hyp-of-smoothing-PI-key-prop-implies-ESSI}
    Suppose for some $m\geq 0$ that a diffeomorphism
    \[\wh{L} \colon X \# W_m \xrasmooth X \# W_m\] satisfies $\wh{L}(\Sigma_1) = \Sigma_2$. Suppose moreover that $\wh{L}$ is smoothly pseudo-isotopic to the identity via a smooth pseudo-isotopy
    \[
    \Xi \colon (X \# W_m) \times I \xrasmooth (X \# W_m) \times I,
    \]
    such that $\Sigma(\Xi) \in \im\big(i_*\colon \Wh_2(\pi_1(X_1)) \to \Wh_2(\pi)\big)$.
  Then after some number of stabilisations of $X$, the surfaces $\Sigma_1$ and $\Sigma_2$ are smoothly isotopic rel.~boundary.
\end{lemma}

\begin{proof}
    Write $y:=\Sigma(\Xi)$ and choose $\ol{y}\in \Wh_2(\pi_1(X_1))$ in the preimage of $-y \in \Wh_2(\pi)$. 
By~\cref{thm:singh}, there exists $n\geq m$ and a smooth pseudo-isotopy
\[
\Lambda\colon (X_1 \# W_n) \times I \xrasmooth (X_1 \# W_n) \times I,
\]
restricting to the identity on $(X_1\#W_n)\times\{0\}$ and $(\partial X_1\#W_n)\times I$,
with $\Sigma(\Lambda)=\ol{y}$. Denote by~$\wh{\Lambda}\colon (X_1 \# W_n) \times I \to (X_1 \# W_n)\times I$ the extension of $\Lambda$ over $(\overline{\nu}\Sigma_1)\times I $ by the identity map. Stabilise $\Xi$ with $\Id_{W_{n-m} \times I}$ to obtain a stabilised smooth pseudo-isotopy, which we continue to denote by  
$\Xi$, now as a diffeomorphism $\Xi\colon (X \# W_n) \times I \to (X \# W_n) \times I$. 
We compute that
\[\Sigma(\Xi \circ \wh{\Lambda}) = \Sigma(\Xi) + \Sigma(\wh{\Lambda}) = y -y=0,
\]
where the first equality uses that $\Sigma$ is a homomorphism and the second is due to the equality~$\Sigma(\wh{\Lambda})=i_*(\Sigma(\Lambda))$ which follows from the definition of the Hatcher--Wagoner obstruction, as~$\wh{\Lambda}$ restricts to an isotopy outside of~$X_1\# W_n$ (indeed, the product isotopy).  

By \cref{thm:PI-iff-stable-smooth-isotopy}, it follows that \[ (\Xi \circ \wh{\Lambda})|_{(X \# W_n) \times \{1\}} =: (\Xi \circ \wh{\Lambda})_1 = 
\wh{L} \circ \wh{\Lambda}_1\] 
is smoothly stably isotopic rel.~boundary to the identity. Also note that both $\wh{\Lambda}_1 \colon X_1 \to X_1$ and $\wh{L} \colon X_1 \to X_2$  are obtained from filling in the surface tubular neighbourhoods, and so $\wh{L} \circ \wh{\Lambda}_1$ sends $\Sigma_1$ to $\Sigma_2$. This completes the proof of the lemma.
\end{proof}

\begin{proof}[Proof of \cref{thm:main}]
Assume that $\Sigma_1$ and $\Sigma_2$ are topologically isotopic rel.\ boundary. Using~\cref{lem:make-homeo-smooth-on-surface-neighbourhood}, we may assume there exists a homeomorphism $\widehat{F}\colon X\to X$ such that the restriction $\wh{F}|_{\ol{\nu}\Sigma_1}\colon \ol{\nu}\Sigma_1\to \ol{\nu}\Sigma_2$ is a diffeomorphism sending $\Sigma_1$ to $\Sigma_2$ and such that $\wh{F}$ is topologically isotopic rel.~boundary to the identity map~$\Id_X$. By~\cref{lem:wh-G-top-PI-to-Id}, there exists $k \geq 0$ and a diffeomorphism 
\[
\wh{G} \colon X\# W_k \xrasmooth X\# W_k
\] 
that sends $\Sigma_1$ to~$\Sigma_2$ and is topologically pseudo-isotopic to $\Id_{X\# W_k}$. Choosing such a topological pseudo-isotopy~$\Psi$, by~\cref{lem:sigma-comes-from-exterior}, we have that $\Sigma^{\Top}(\Psi) \in \im\big(i_*\colon \Wh_2(\pi_1(X_1)) \to \Wh_2(\pi)\big)$.

Next, by \cref{prop:aim-of-smoothing-PI-section}, assuming the hypotheses of \cref{thm:main}, we obtain $m \geq k$ and a diffeomorphism 
\[\wh{L} \colon X\# W_m \xrasmooth X\# W_m\] that sends $\Sigma_1$ to~$\Sigma_2$ and is smoothly pseudo-isotopic to $\Id_{X\# W_m}$, via a smooth pseudo-isotopy~$\Xi$ such that $\Sigma(\Xi) \in \im\big(i_*\colon \Wh_2(\pi_1(X_1)) \to \Wh_2(\pi)\big)$.
These are exactly the hypotheses of \cref{lemma:hyp-of-smoothing-PI-key-prop-implies-ESSI}, and so that lemma implies that after some number of stabilisations of $X$, the surfaces $\Sigma_1$ and $\Sigma_2$ are smoothly isotopic rel.~boundary.
\end{proof}

\appendix

\section{The case of simply-connected $X$ and simply-connected complement}
\label{sec:appendix}

The case of Theorem~\ref{thm:main} when $X$ is simply-connected was first proved in~\cite{GalvinCS}*{Theorem~1.2}. Specialising even more, to the case when moreover the surface complements are simply-connected, we are able to give an alternative proof based on work of Gompf~\cite{Gompf-stable}, Boyer~\cite{Boyer93}, 
Saeki~\cite{Saeki}, and Quinn~\cite{Quinn:isotopy} (with the correction in~\cite{GGHKP}).

\begin{proposition}
\label{prop:saekiproof}
  Let $X$ be a smooth, orientable, closed, simply-connected $4$-manifold.
    Suppose~$\Sigma_1, \Sigma_2\subseteq X$ are smooth, orientable, closed surfaces, each with simply-connected complements. Suppose that~$\Sigma_1$ and $\Sigma_2$ are topologically isotopic. 
    Then $\Sigma_1$ and $\Sigma_2$ are smoothly isotopic in some stabilisation of $X$.
    \end{proposition}

    \begin{proof}
        Using~\cref{lem:make-homeo-smooth-on-surface-neighbourhood}, we may assume there exists a homeomorphism $\widehat{F}\colon X\to X$ such that the restriction $\wh{F}|_{\ol{\nu}\Sigma_1}\colon \ol{\nu}\Sigma_1\to \ol{\nu}\Sigma_2$ is a diffeomorphism sending $\Sigma_1$ to $\Sigma_2$ and such that $\wh{F}$ is topologically isotopic rel.~boundary to the identity map~$\Id_X$. The restriction of $\wh{F}$ to the surface exteriors is a homeomorphism, which we denote $F\colon X_1\to X_2$.
By Gompf's theorem~\cite{Gompf-stable}, there exists a~$k\geq0$ and a diffeomorphism $G\colon X_1\#W_k\to X_2\#W_k$ that agrees with $F$ on the boundary (Gompf's theorem is not stated rel.~boundary, but from his proof it is clear this conclusion is available). Write $\wh{G}\colon X\to X$ for the diffeomorphism obtained by extending $G$ over the tubular neighbourhoods of the surfaces, using the diffeomorphism $\wh{F}|_{\ol{\nu}\Sigma_1}\colon \ol{\nu}\Sigma_1\to \ol{\nu}\Sigma_2$.

In general, a homeomorphism of a compact $4$-manifold $(f,\Id_{\partial M}) \colon (M,\partial M)\to (M,\partial M)$ determines what is called in~\cite{Orson-Powell}*{\textsection2.2} a \emph{Poincar\'{e} variation}. This is a certain homomorphism
\[
\Delta_f\colon H_2(M,\partial M)\to H_2(M)
\]
enjoying favourable interactions with Poincar\'{e} duality and algebraic properties; see~\cite{Saeki} and~\cite{Orson-Powell} for more details. We will only need the property that a Poincar\'{e} variation recovers the ordinary map on absolute homology as
\[
f_*=(\Id-\Delta_f\circ j)\colon H_2(M)\xrightarrow{\cong} H_2(M),
\]
where $j$ is the map in the long exact sequence of the pair.
Consider the Poincar\'{e} variation
\begin{equation}
    \label{eq:variation}
\Delta_{F\circ G^{-1}}\colon H_2(X_2,\partial X_2)\to H_2(X_2)
\end{equation}
induced by
$(F\circ G^{-1}, \Id_{\partial X_2})\colon (X_2,\partial X_2)\to (X_2,\partial X_2)$.
By~\cite{Saeki}*{Theorem~3.7}, there exists~$k\geq 0$ and a diffeomorphism $J\colon X_2\#W_k\to X_2\#W_k$ with $J|_{\partial{X_2\#W_k}}=\Id_{\partial X_2\#W_k}$ realising the variation~\eqref{eq:variation} smoothly stably
\[
\Delta_J=\Delta_{(F\circ G^{-1})\#\Id_{W_k}}=\Delta_{(F\#\Id_{W_k})\circ( G^{-1}\#\Id_{W_k})}.
\]
In particular, we have that $J$ induces the homomorphism
\[
(F\#\Id_{W_k})_*\circ( G\#\Id_{W_k})^{-1}_*\colon H_2(X_2\#W_k)\to H_2(X_2\#W_k).
\]
Write $\widehat{J}\colon X\#W_k\to X\#W_k$ for the extension of $J$ by the identity.

We now have a diffeomorphism
\[
\wh{K}=\wh{J}\circ\wh{G}\colon X\# W_k\xrasmooth X\# W_k.
\]
that agrees agrees with $\wh{F}$ on restriction to $\overline{\nu}\Sigma_1$. Moreover, by construction, we have
\begin{equation}
    \label{eq:remember}
\wh{K}_*=(\wh{F}\#\Id_{W_k})_*\colon H_2(X \# W_k)\to H_2(X\# W_k).
\end{equation}

Components of the argument made  in Boyer~\cite{Boyer93}*{\textsection 4} now suffice to show $\wh{K}$ induces the identity map on~$H_2(X\# W_k)$. We sketch the argument for the convenience of the reader. Set $E:=\{x\in H_2(X\# W_k)\,|\,x\cdot[\Sigma_1]=0\}$. By ~\cite{Boyer93}*{Lemma~4.1}, we have $E=\im(\varphi_j)$ for $j=1,2$, where $\varphi_j\colon H_2(X_j\# W_k)\to H_2(X\# W_k)$ is the inclusion-induced map. We reproduce part of~\cite{Boyer93}*{Figure~4.2}, below, which is commutative diagram where the horizontal sequences are exact. Note, Boyer is using the diagram to construct a homeomorphism of the exteriors extending the map on the boundary, whereas we already have one, so the diagrams look a little different, but do agree.
\[
\begin{tikzcd}
0 \ar[r]
& H_3(X\# W_k,X_1\# W_k) \ar[r]\ar[ddd,"\wh{K}_*"', bend right=80]
& H_2(X_1\# W_k)\ar[r, "\varphi_1"] \ar[ddd,"\wh{K}_*"']
& E \ar[r] \ar[ddd, "\wh{K}_*"']
& 0
\\
& H_3(\overline{\nu}\Sigma_1,\partial (X_1\# W_k))\ar[u,"\cong"']\ar[d, "(\wh{F}\#\Id_{W_k})_*"']
&&&
\\
&H_3(\overline{\nu}\Sigma_2,\partial (X_2\# W_k))\ar[d,"\cong"]
&&&
\\
0 \ar[r]
& H_3(X\# W_k,X_2\# W_k) \ar[r]
& H_2(X_2\# W_k)\ar[r, "\varphi_2"]
& E \ar[r]
& 0
\end{tikzcd}
\]
A similar diagram could be produced, using the homeomorphism $\wh{F}\# \Id_{W_k}$ in place of $\wh{K}$. Using~\eqref{eq:remember}, we observe the left and central columns of such a diagram would then have the same maps as the left and central columns in our diagram. In a map of short exact sequences, the left and central vertical maps uniquely determine the right vertical map. Hence we conclude $\wh{K}$ and $\wh{F}\# \Id_{W_k}$ induce the same map on~$E$. But $\wh{F}$ is topologically isotopic to the identity, and hence so is $\wh{F}\# \Id_{W_k}$. Thus the right vertical map in the diagram above is $\Id_E$.

In the proof of~\cite{Boyer93}*{Lemma~4.5}, Boyer argues that the fact  $\wh{K}_*\colon H_2(X\# W_k)\to H_2(X\# W_k)$ is the identity map on $[\Sigma_1]$ and on $E$ is enough to conclude it is the identity map overall. Briefly, in the case that the homology class $[\Sigma_i]$ has nontrivial self-intersection, this follows because $[\Sigma]$ and $E$ rationally generate $H_2(X\# W_k)$. In the case of trivial self-intersection $H_2(X\# W_k)$ is integrally generated by $E$ and an algebraic dual to $[\Sigma_1]$, which Boyer argues is enough for the conclusion. We refer the reader to~\cite{Boyer93}*{Lemma~4.5} for full details of the argument.
 
By Quinn's theorem~\cite{Quinn:isotopy} (with the correction in~\cite{GGHKP}), or alternatively \cref{thm:PI-iff-stable-smooth-isotopy}, we thus have, possibly after further stabilisations of $X\#W_k$, that~$\widehat{K}$ is smoothly isotopic to the identity.
\end{proof}

\begin{remark}
\label{rem:AKMR}
In the introduction of~\cite{AKMR}, it was asserted that the combination of work Wall~\cite{MR163323}, Perron~\cite{Perron} and Quinn~\cite{Quinn:isotopy} would show homologous $2$-spheres with simply-connected complement are topologically isotopic, and become smoothly isotopic in some stabilisation. 
In~\cite{AKMRS} it was asserted that the result of~\cref{prop:saekiproof} would follow from Perron~\cite{Perron} and Quinn~\cite{Quinn:isotopy}.  We believe that in fact citations to Boyer and Saeki are also needed to make these arguments, as we explain next.

The statement that homologous surfaces with simply-connected complements in a simply-connected $4$-manifold are topologically isotopic is a result due to Boyer~\cite{Boyer93}*{Theorem~F}. In the proof, Boyer indeed appeals to the Perron--Quinn result that homeomorphisms of closed simply-connected 4-manifolds inducing the identity on second homology are topologically isotopic to the identity, however there is much work to be done before he can invoke that result.

For the statements that topological isotopy implies smooth stable isotopy, we think both implied proofs were intended to follow the structure of the proof we gave in~\cref{prop:saekiproof}. We suspect Wall's theorem~\cite{MR163323} was intended to be used at the place where we invoked Saeki~\cite{Saeki}. Wall's theorem is insufficient to modify a diffeomorphism of a manifold with boundary, so this theorem cannot be applied in that way. Alternatively, one might try to apply Wall's theorem after the surfaces have been filled back in, but this might void the carefully arranged condition that the diffeomorphism of the stabilised manifold sends $\Sigma_1$ to $\Sigma_2$.  So Saeki's generalisation of Wall's theorem to the case of nonempty boundary seems to be necessary. 
\end{remark}

\def\MR#1{}
\bibliography{ESSI-bib}

@article{Stong-Wang,
 author = {Stong, Richard and Wang, Zhenghan},
 title = {Self-homeomorphisms of 4-manifolds with fundamental group {{\(\mathbb{Z}\)}}},
 fjournal = {Topology and its Applications},
 journal = {Topology Appl.},
 issn = {0166-8641},
 volume = {106},
 number = {1},
 pages = {49--56},
 year = {2000}}

@unpublished{galvin_2024,
Author = {Galvin, Daniel A. P.},
Year = {2024},
Title = {Ronnie {L}ee's generator for {{\(L_5(\mathbb{Z}[\mathbb{Z}])\)}}},
note = {Preprint, available on author's website at \url{https://dapgalvin.github.io/ronnie_lee_note.pdf}}
}

@unpublished{Galvin:2025aa,
	author = {Daniel Galvin and Isacco Nonino},
	note = {Preprint, available at ar{X}iv:2506.11905},
	title = {Pseudo-isotopy versus isotopy for homeomorphisms of 4-manifolds},
	url = {https://arxiv.org/pdf/2506.11905.pdf},
	year = {2025}}

@unpublished{OPRW,
	author = {Orson, Patrick and Powell, Mark and Randal-Williams, Oscar},
	note = {Preprint, available at ar{X}iv:2507.16984},
	title = {Smoothing topological pseudo-isotopies of 4-manifolds},
	url = {https://arxiv.org/pdf/2507.16984.pdf},
	year = {2025}}

@article {Saeki,
    AUTHOR = {Saeki, Osamu},
     TITLE = {Stable mapping class groups of 4-manifolds with boundary},
   JOURNAL = {Trans. Amer. Math. Soc.},
  FJOURNAL = {Transactions of the American Mathematical Society},
    VOLUME = {358},
      YEAR = {2006},
    NUMBER = {5},
     PAGES = {2091--2104},
}

@article {AKMR,
    AUTHOR = {Auckly, Dave and Kim, Hee Jung and Melvin, Paul and Ruberman,
              Daniel},
     TITLE = {Stable isotopy in four dimensions},
   JOURNAL = {J. Lond. Math. Soc. (2)},
  FJOURNAL = {Journal of the London Mathematical Society. Second Series},
    VOLUME = {91},
      YEAR = {2015},
    NUMBER = {2},
     PAGES = {439--463}
}

@article {AKMRS,
    AUTHOR = {Auckly, Dave and Kim, Hee Jung and Melvin, Paul and Ruberman,
              Daniel and Schwartz, Hannah},
     TITLE = {Isotopy of surfaces in 4-manifolds after a single
              stabilization},
   JOURNAL = {Adv. Math.},
  FJOURNAL = {Advances in Mathematics},
    VOLUME = {341},
      YEAR = {2019},
     PAGES = {609--615},
      ISSN = {0001-8708,1090-2082},
   MRCLASS = {57R52},
  MRNUMBER = {3873547},
MRREVIEWER = {Sergey\ M.\ Finashin},
       DOI = {10.1016/j.aim.2018.10.040},
       URL = {https://doi.org/10.1016/j.aim.2018.10.040},
}

@article{Gompf-stable,
	author = {Gompf, Robert E.},
	fjournal = {Topology and its Applications},
	journal = {Topology Appl.},
	number = {2-3},
	pages = {115--120},
	title = {Stable diffeomorphism of compact {$4$}-manifolds},
	volume = {18},
	year = {1984}}

@article{chakim-lightbulb,
	author = {Jae Choon Cha and Byeorhi Kim},
	journal = {J. Eur. Math. Soc., published online first},
	title = {Light bulb smoothing for topological surfaces in {4}-manifolds},
    doi = {10.4171/JEMS/1731},
	year = {2025}}

@article{GalvinCS,
	author = {Daniel A. P. Galvin},
	journal = {Arxiv 2405.07928},
	title = {The {C}asson-{S}ullivan invariant for homeomorphisms of {4}-manifolds},
	year = {2024}}

@article{MR163324,
	author = {Wall, C. Terence C.},
	doi = {10.1112/jlms/s1-39.1.141},
	fjournal = {The Journal of the London Mathematical Society},
	issn = {0024-6107},
	journal = {J. London Math. Soc.},
	mrclass = {57.10},
	mrnumber = {163324},
	mrreviewer = {N. Kuiper},
	pages = {141--149},
	title = {On simply-connected {$4$}-manifolds},
	url = {https://doi.org/10.1112/jlms/s1-39.1.141},
	volume = {39},
	year = {1964}}

@article{MR163323,
	author = {Wall, C. Terence C.},
	doi = {10.1112/jlms/s1-39.1.131},
	fjournal = {The Journal of the London Mathematical Society},
	issn = {0024-6107},
	journal = {J. London Math. Soc.},
	mrclass = {57.20},
	mrnumber = {163323},
	mrreviewer = {N. Kuiper},
	pages = {131--140},
	title = {Diffeomorphisms of {$4$}-manifolds},
	url = {https://doi.org/10.1112/jlms/s1-39.1.131},
	volume = {39},
	year = {1964}}

@incollection {MR1434103,
    AUTHOR = {Sullivan, D. P.},
     TITLE = {Triangulating and smoothing homotopy equivalences and
              homeomorphisms. {G}eometric {T}opology {S}eminar {N}otes},
 BOOKTITLE = {The {H}auptvermutung book},
    SERIES = {$K$-Monogr. Math.},
    VOLUME = {1},
     PAGES = {69--103},
 PUBLISHER = {Kluwer Acad. Publ., Dordrecht},
      YEAR = {1996},
      ISBN = {0-7923-4174-0},
   MRCLASS = {57Q15 (57Q25 57R10)},
  MRNUMBER = {1434103},
MRREVIEWER = {Oliver\ Attie},
       DOI = {10.1007/978-94-017-3343-4\{_}3}

@unpublished{KK24,
	author = {Manuel Krannich and Alexander Kupers},
	note = {Preprint, available at ar{X}iv:2402.16167},
	title = {A note on homotopy and pseudoisotopy of diffeomorphisms of 4-manifolds},
	year = {2024}}

@unpublished{GGHKP,
	author = {David Gabai and David {T.} Gay and Daniel Hartman and Vyachslav Krushkal and Mark Powell},
	note = {Preprint, available at ar{X}iv:2311.11196},
	title = {Pseudo-isotopies of simply connected 4-manifolds},
	year = {2023}}

@article{Boyer93,
	author = {Boyer, Steven},
	fjournal = {Commentarii Mathematici Helvetici},
	issn = {0010-2571},
	journal = {Comment. Math. Helv.},
	number = {1},
	pages = {20--47},
	title = {Realization of simply-connected 4-manifolds with a given boundary},
	volume = {68},
	year = {1993}}

@unpublished{Gabai-22,
	author = {David Gabai},
	note = {Preprint, available at ar{X}iv:2212.02004},
	title = {$3$-Spheres in the $4$-Sphere and Pseudo-Isotopies of ${S}^1 \times {S}^3$},
	year = {2022}}

@article {Orson-Powell,
    AUTHOR = {Orson, Patrick and Powell, Mark},
     TITLE = {Mapping class groups of simply connected 4-manifolds with
              boundary},
   JOURNAL = {J. Differential Geom.},
  FJOURNAL = {Journal of Differential Geometry},
    VOLUME = {131},
      YEAR = {2025},
    NUMBER = {1},
     PAGES = {199--275},
      ISSN = {0022-040X},
   MRCLASS = {57},
  MRNUMBER = {4947553},
       DOI = {10.4310/jdg/1755544135},
       URL = {https://doi.org/10.4310/jdg/1755544135},
}

@article{Singh,
 author = {Singh, Oliver},
 title = {Pseudo-isotopies and diffeomorphisms of 4-manifolds},
 fjournal = {Journal of Topology},
 journal = {J. Topol.},
 issn = {1753-8416},
 volume = {18},
 number = {4},
 pages = {61},
 note = {Id/No e70043},
 year = {2025}}

@article{Quinn:isotopy,
	author = {Quinn, Frank},
	fjournal = {Journal of Differential Geometry},
	journal = {J. Differential Geom.},
	number = {3},
	pages = {343--372},
	title = {Isotopy of {$4$}-manifolds},
	volume = {24},
	year = {1986}}

@book{Kirby-Siebenmann:1977-1,
	address = {Princeton, N.J.},
	author = {Kirby, Robion C. and Siebenmann, Laurence C.},
	note = {With notes by John Milnor and Michael Atiyah, Annals of Mathematics Studies, No. 88},
	pages = {vii+355},
	publisher = {Princeton University Press},
	title = {Foundational essays on topological manifolds, smoothings, and triangulations},
	year = {1977}}

@article{F,
	author = {Freedman, Michael},
	fjournal = {Journal of Differential Geometry},
	journal = {J. Differential Geom.},
	number = {3},
	pages = {357--453},
	title = {The topology of four-dimensional manifolds},
	volume = {17},
	year = {1982}}

@incollection {MR1434102,
    AUTHOR = {Casson, Andrew J.},
     TITLE = {Generalisations and applications of block bundles},
 BOOKTITLE = {The {H}auptvermutung book},
    SERIES = {$K$-Monogr. Math.},
    VOLUME = {1},
     PAGES = {33--67},
 PUBLISHER = {Kluwer Acad. Publ., Dordrecht},
      YEAR = {1996},
      ISBN = {0-7923-4174-0},
   MRCLASS = {57Q50},
  MRNUMBER = {1434102},
MRREVIEWER = {Oliver\ Attie},
       DOI = {10.1007/978-94-017-3343-4\{_}2}

@book{FQ,
	author = {Freedman, Michael and Quinn, Frank},
	pages = {viii+259},
	place = {Princeton, NJ},
	publisher = {Princeton University Press},
	series = {Princeton Mathematical Series},
	title = {Topology of $4$-manifolds},
	volume = {39},
	year = {1990}}

@book{Br93,
	author = {Bredon, Glen E.},
	isbn = {0-387-97926-3},
	mrclass = {55-01 (54-01 57-01)},
	mrnumber = {1700700},
	note = {Corrected third printing of the 1993 original},
	pages = {xiv+557},
	publisher = {Springer-Verlag, New York},
	series = {Graduate Texts in Mathematics},
	title = {Topology and geometry},
	volume = {139},
	year = {1997}}

@article{Perron,
	author = {Perron, Bernard},
	fjournal = {Topology. An International Journal of Mathematics},
	journal = {Topology},
	number = {4},
	pages = {381--397},
	title = {Pseudo-isotopies et isotopies en dimension quatre dans la cat\'{e}gorie topologique},
	volume = {25},
	year = {1986}}

@inproceedings{Kreck-isotopy-classes,
	author = {Kreck, Matthias},
	booktitle = {Algebraic topology, {A}arhus 1978 ({P}roc. {S}ympos., {U}niv. {A}arhus)},
	pages = {643--663},
	publisher = {Springer, Berlin},
	series = {Lecture Notes in Math.},
	title = {Isotopy classes of diffeomorphisms of {$(k-1)$}-connected almost-parallelizable {$2k$}-manifolds},
	volume = {763},
	year = {1979}}

\end{document}